\documentclass[bj,preprint]{imsart}

\RequirePackage{amsthm,amsmath,amsfonts,amssymb}
\RequirePackage[numbers]{natbib}

\usepackage{enumerate}

\usepackage{todonotes}
\presetkeys%
{todonotes}%
{inline}{}

\newcommand{\E}{\mathbb{E}}

\newcommand{\Po}{\mathbb{P}}

\newcommand{\J}{\mathcal{J}}

\newcommand{\pa}[1]{\left(#1\right)}
\newcommand{\cro}[1]{\left[#1\right]}
\newcommand{\ac}[1]{\left\{#1\right\}}
\newcommand{\ab}[1]{\left| #1 \right|}    

\newcommand{\Var}{\mbox{Var}}

\newcommand{\un} {\mbox{\rm 1\hspace{-0.3em}I}}
\newcommand{\PJ}{\Pi_{S_L}}

\startlocaldefs
\theoremstyle{plain}

\newtheorem{theorem}{\textbf{Theorem}}[section]

\newtheorem{remark}[theorem]{\textbf{Remark}}

\newtheorem{lemma}[theorem]{\textbf{Lemma}}%[section]

\newtheorem{corollary}[theorem]{\textbf{Corollary}}%[section]
\newcommand{\co}[1]{#1}

\endlocaldefs

\begin{document}
	
	\begin{frontmatter}
		\title{Minimax optimal goodness-of-fit testing for densities and multinomials under a local differential privacy constraint}
		%\title{A sample article title with some additional note\thanksref{T1}}
		\runtitle{Optimal private testing}
		%\thankstext{T1}{A sample of additional note to the title.}
		
		\begin{aug}
			\author[A]{\fnms{Joseph} \snm{Lam-Weil}\ead[label=e1]{joseph.lam@ovgu.de}},
			\author[B]{\fnms{B\'{e}atrice} \snm{Laurent}\ead[label=e2]{laurentb@insa-toulouse.fr}}
			\and
			\author[C]{\fnms{Jean-Michel} \snm{Loubes}\ead[label=e3]{loubes@math.univ-toulouse.fr}}
			
			\runauthor{J. Lam-Weil et al.}
			
			%%%%%%%%%%%%%%%%%%%%%%%%%%%%%%%%%%%%%%%%%%%%%%
			%% Addresses                                %%
			%%%%%%%%%%%%%%%%%%%%%%%%%%%%%%%%%%%%%%%%%%%%%%
			\address[A]{Magdeburg University, Germany, \printead{e1}}
			
			\address[B]{INSA de Toulouse, France, \printead{e2}}
			
			\address[C]{Universit\'{e} de Toulouse, France, \printead{e3}}
		\end{aug}
		
		\begin{abstract}
			Finding anonymization mechanisms to protect personal data is at the heart of recent machine learning research. Here, we consider the consequences of local differential privacy constraints on  goodness-of-fit testing, i.e. the statistical problem assessing whether sample points are generated from a fixed density $f_0$, or not. The observations are kept hidden and replaced by a stochastic transformation satisfying the local differential privacy constraint. In this setting, we propose a testing procedure which is based on an estimation of the quadratic distance between the density $f$ of the unobserved samples and $f_0$. \co{We establish an upper bound on the separation distance associated with this test, and a matching lower bound on the minimax separation rates of testing under non-interactive privacy in the case that $f_0$ is uniform, in discrete and continuous settings.} To the best of our knowledge, we provide the first minimax optimal test and associated private transformation under a local differential privacy constraint over Besov balls in the continuous setting, quantifying the price to pay for data privacy. We also present a test that is adaptive to the smoothness parameter of the unknown density and remains minimax optimal up to a logarithmic factor. Finally, we note that our results can be translated to the discrete case, where the treatment of probability vectors is shown to be equivalent to that of piecewise constant densities in our setting. That is why we work with a unified setting for both the continuous and the discrete cases.
		\end{abstract}
		
		\begin{keyword}
			\kwd{local differential privacy}
			\kwd{non-interactive privacy}
			\kwd{goodness-of-fit testing}
			\kwd{minimax separation rates}
			\kwd{continuous and discrete distributions}
		\end{keyword}
		
	\end{frontmatter}
	
	%%%%%%%%%%%%%%%%%%%%%%%%%%%%%%%%%%%%%%%%%%%%%%
	%%%% Main text entry area:                  %%
	%%%%%%%%%%%%%%%%%%%%%%%%%%%%%%%%%%%%%%%%%%%%%%

	\section{Introduction}
	Ensuring user privacy is at the core of the development of Artificial Intelligence. Indeed datasets can contain extremely sensitive information, and someone with access to a privatized training set or the outcome of an algorithm should not be able to retrieve the original dataset. However, classical anonymization and cryptographic approaches fail to prevent the disclosure of sensitive information in the context of learning. Indeed, with the example of a hospital's database, removing names and social security numbers from databases does not prevent the identification of patients using a combination of other attributes like gender, age or illnesses. \cite{dinur2003revealing} cites Cystic Fibrosis as an example which exist with a frequency of around 1/3000. Hence differential privacy mechanisms were developed to cope with such issues.	Such considerations can be traced back to \cite{warner1965randomized,duncan1986disclosure,duncan1989risk,fienberg1998disclosure}. As early as in 1965, \cite{warner1965randomized} presented the first privacy mechanism which is now a baseline method for binary data: Randomized response. Another important result is presented in the works of \cite{duncan1986disclosure,duncan1989risk,fienberg1998disclosure}, where they expose a trade-off between statistical utility, or in other terms performance, and privacy in a limited-disclosure setting.
	
	Differential privacy as expressed in \cite{dwork2006calibrating,dwork2006our} is the most common formalization of the problem of privacy. It can be summed up as the following condition: altering a single data point of the training set only affects the probability of an outcome to a limited degree. One main advantage of such a definition of privacy is that it can be parametrized by some positive parameter $\alpha$, where $\alpha$ close to 0 corresponds to a more restrictive privacy condition. This definition treats privacy in a global way with respect to the original dataset, in contrast with the privacy constraint that follows.
	
	We now consider a stronger privacy condition called local differential privacy which also depends on a positive parameter $\alpha$ and where the analyst himself is not trusted with the data.
	Consider unobserved random variables $X_1,\ldots,X_n$ taking values in $[0,1]$, which are independent and identically distributed (i.i.d.) with density $f $ with respect to the Lebesgue measure.
	We observe $Z_1,\ldots,Z_n$ which are $\alpha$-local differentially private views of $X_1,\ldots,X_n$. That is, there exist probability measures $Q_1,\ldots,Q_n$ such that for all $0\leq i\leq n$, $Z_i$ is a stochastic transformation of $X_i$ by the channel $Q_i$ and
	\begin{equation}
		\label{eq:localPrivacy}
		\sup_{S \in \mathcal Z_i, z_j \in \tilde \Omega_j, (x, x') \in [0,1]^2} \frac{Q_i(Z_i \in S|X_i=x, Z_j = z_j, j < i)}{Q_i(Z_i \in S|X_i=x', Z_j = z_j, j<i)} \leq e^\alpha,
	\end{equation}
	where $Q_i(Z_i \in \tilde \Omega_i) = 1$ and $\mathcal Z_i$ is a $\sigma$-algebra such that $\tilde \Omega_i$ is its associated sample space.	
	This notion has been extensively studied through the concept of local algorithms, especially in the context of privacy-preserving data mining \cite{warner1965randomized, agrawal2000privacy,agrawal2001design,van2002randomized, evfimievski2003limiting, agrawal2005framework, mishra2006privacy, jank2008statistical, kasiviswanathan2011can}. Now note that Equation~\eqref{eq:localPrivacy} accounts for possible dependencies between $Z_i$'s, corresponding to the interactive case. The role of interactivity has been further studied in \cite{joseph2019role,butucea2020interactive,berrett2020locally}, and it can be complete or sequential. Recent results detailed in \cite{duchi2013locallower, duchi2013localproba, duchi2013localcomplete} give information processing inequalities depending on the local privacy constraint via the parameter $\alpha$. Those can be used to obtain Fano or Le Cam-type inequalities in order to obtain a minimax lower bound for estimation or testing problems. Our proof also relies on Le Cam's inequality, albeit in a more refined way in order to obtain minimax optimal results.

	Testing problems have appeared as crucial tools in machine learning in order to assess whether a model fits the observations or to detect anomalies and novelties.  In particular,
	%, given some known density, goodness-of-fit measures the discrepancy between observed values and that density.
	goodness-of-fit testing is a classical hypothesis testing problem in statistics. It consists in testing whether the density $f$ of $n$ independent and identically distributed (i.i.d.)~observations  equals  a specified density $f_0$ or not. This motivates our study of goodness-of-fit testing under a local differential privacy constraint.
	
	We want to design our tests so that they reject the null hypothesis $H_0: f= f_0$ if the data is not actually generated from the given model with a given confidence level. 	 Assuming that $ f $ and $f_0$ belong to $ \mathbb{L}_2([0,1])= \ac{f : [0,1]\rightarrow \mathbb{R}, \|f\|_2^2= \int_0^1 f^2(x) dx < \infty}$, it is natural to propose a test based on an estimation of the squared $\mathbb{L}_2$-distance $ \|f-f_0\|_2^2$  between $f$ and $f_0$. In order to test whether $f=f_0$ from the observation of an i.i.d sample set $(X_1, \ldots, X_n)$ with common density $f$, \cite{Neyman1937} introduces an orthonormal basis $\{f_0, \phi_l, l \geq 0 \} $ of $ \mathbb{L}_2([0,1])$. The goodness-of-fit hypothesis is rejected if the estimator $ \sum_{l=1}^D (\sum_{i=1}^n \phi_l(X_i)/n)^2$ exceeds some threshold, where $D$ is a given integer depending on $n$. Data-driven versions of this test, where the parameter $D$ is chosen to minimize some penalized criterion have been introduced by  \cite{BR1992,Ledwina1994,KalLedwina1995,IngLedwina1996}. 
	
	Additionally, we want to find  the limitations of a test by determining how close the two hypotheses can get while remaining separated by the testing procedure.
	This classical problem has been studied under the lens of minimax optimality in the seminal work by \cite{ingster1987minimax,ingster1993}. Non-asymptotic performances and an extension to composite null hypotheses are provided in \cite{FL06}. 
	In order to introduce the notion of minimax optimality for a testing procedure, let us recall some definitions.
	We consider the uniform separation rate as defined in \cite{Baraud2002}. Let $ \Delta_{\gamma}$ be a $\gamma$-level test  with values in $\{0,1\}$, where $\Delta_{\gamma}=1$ corresponds to the decision of rejecting the null hypothesis $ f=f_0$ and $\Po_{f_0}(\Delta_\gamma = 1) \leq \gamma$. The uniform separation rate $\tilde \rho_n \left( \Delta_{\gamma} , \mathcal{C} , \beta, f_0 \right)$ of the test $ \Delta_{\gamma}$ with respect to the $\mathbb{L}_2$-norm, over a class $\mathcal{C}$ of alternatives $f$ such that $f- f_0$ satisfies smoothness assumptions, is defined for all $\beta$ in $(0,1)$ as
	\begin{align}\label{seprate}
		&\tilde \rho_n \left(  \Delta_{\gamma}, \mathcal C, \beta, f_0 \right) = \inf \Big\{ \rho > 0;  \sup_{f \in \mathcal C, \left\| f-f_0 \right\|_2 > \rho } \mathbb{P}_{f} \left(  \Delta_{\gamma}(X_1, \ldots, X_n) = 0 \right) \leq \beta  \Big\}, 
	\end{align} 
	where $ \mathbb{P}_{f} $ denotes the distribution of the i.i.d.~samples $(X_1, \ldots, X_n)$ with common density $f$.
	
	The uniform separation rate is then the smallest value in the sense of the $\mathbb{L}_2$-norm of $(f - f_0)$  for which the second kind error of the test is uniformly controlled by  $\beta$ over $  \mathcal C$. This definition extends the notion of critical radius introduced in  \cite{ingster1993} to the non-asymptotic framework.  Note that minimax separation rates are at least as fast as minimax estimation rates and the interest lies in determining problems where testing can be done faster than estimating.
	
	%We recall that a test is of level $\gamma$, if $\Po_{f_0}(\Delta_\gamma = 1) \leq \gamma.$
	A test with level $\gamma$ having optimal performances should then have the smallest possible uniform separation rate (up to a multiplicative constant) over $\mathcal C$. 
	To quantify this,  \cite{Baraud2002} introduces the non-asymptotic minimax rate of testing defined by 
	\begin{equation}\label{minimaxrate}
		\tilde \rho_n^*\left( \mathcal{C}, \gamma, \beta, f_0 \right) = \inf_{\Delta_{\gamma}} \tilde \rho_n \left(  \Delta_{\gamma}, \mathcal{C}, \beta, f_0 \right),
	\end{equation}
	where the infimum is taken over all tests of level  $\gamma$. 
	A test is  optimal in the minimax sense over the class $\mathcal{C}$ if its uniform separation rate  is upper-bounded, up to some constant, by the non-asymptotic minimax rate of testing. Taking $\mathcal{C}$ too general leads to trivial rates. That is the reason why we restrict our study to two cases. On the one hand, we consider multinomial distributions which cover the discrete case. On the other hand, we work in the continuous case with Besov balls, which have been widely used in statistics since the seminal paper by \cite{donoho1996density}. Non-private results already exist for such sets, which make them meaningful for comparisons. Another motivation is that Besov sets are function classes parametrized by smoothness parameters and the minimax rates depend exclusively on those parameters in a lot of problems. Finally, thanks to their interesting properties from approximation theory, a large variety of signals can be dealt with, especially those built using wavelet bases.
	
	We present a few non-private results from the literature. For H\"older classes with smoothness parameter $s>0$, \cite{ingster1993} establishes the asymptotic minimax rate of testing $ n^{-2s/(4s+1)}$.
	The test proposed in their paper is not adaptive since it makes use of a known smoothness parameter $s$. Minimax optimal adaptive goodness-of-fit tests over  H\"older or Besov classes of alternatives  are provided in \cite{ingster2000} and \cite{FL06}. These tests achieve the separation rate $  (n/\sqrt{\log\log(n)})^{-2s/(4s+1)}  $ over a wide range of regularity classes (H\"older or Besov balls) with smoothness parameter $s>0$. The $ \log\log(n) $ term is the optimal price to pay for adaptation to the unknown parameter $s>0$. 
	
	In the discrete case, the goal is to distinguish between $d$-dimensional probability vectors $p$ and $p^0$ using samples from the multinomial distribution with parameters $p$ and $n$. \cite{paninski2008coincidence} obtain that the minimax optimal rate with respect to the $l_1$-distance, $\sum_{i = 1}^d |p_i-p^0_{i}|$, is $d^{1/4}/\sqrt n$. An extension is the study of local minimax rates as in \cite{valiant2017automatic}, where the rate is made minimax optimal for any $p^0$ instead of just in the worst choice of $p^0$. Finally, \cite{balakrishnan2017hypothesis} presents local minimax rates of testing both in the discrete and continuous cases.

	%
	%\todo{cite minimax identity testing. Pay attention to density or multinomial, and L1 or L2.} Minimax optimality allows for a characterization of the difficulty of a problem. In order to get an intuition behind the minimax point of view, consider some problem and an associated value function which characterizes how well any method does in any case. Then minimax optimality corresponds to the value for the best method on the worst case.
	%
	
	A few problems have already been tackled in order to obtain minimax rates under local privacy constraint. The main question is whether the minimax rates are affected by the local privacy constraint and to quantify the degradation of the rate in that case. We define a sample degradation of $C(\alpha)$ in the following way. If $n$ is the necessary and sufficient sample size in order to solve the classical non-private version of a problem, the $\alpha$-local differential private problem is solved with $nC(\alpha)$ samples. For a few problems, a degradation of the effective sample size by a multiplicative constant is found. In \cite{duchi2013localproba}, they obtain minimax estimation rates for multinomial distributions in dimension $d$ and find a sample degradation of $\alpha^2/d$. In \cite{duchi2018minimax}, they also find a multiplicative sample degradation of $\alpha^2 / d$ for generalized linear models, and $\alpha^2$ for median estimation. However, in other problems, a polynomial degradation is noted. For one-dimensional mean estimation, the usual minimax rate is $n^{-(1 \wedge (2-2/k))}$, whereas the private rate from \cite{duchi2018minimax} is $(n\alpha^2)^{-(0 \wedge (1-1/k))}$ for original observations $X$ satisfying $\E(X) \in [-1,1]$ and $\E(|X|^k) < \infty$. As for the problem of nonparametric density estimation presented in \cite{duchi2018minimax}, the rate goes from $n^{-2s/(2s+1)}$ to $(n \alpha^2)^{-2s/(2s+2)}$ over an elliptical Sobolev space with smoothness $s$. This result was extended in \cite{butucea2019local} over Besov ellipsoids. The classical minimax mean squared errors were presented in \cite{yu1997assouad,yang1999information,tsybakov2004introduction}. 
	
	Goodness-of-fit testing has been studied extensively under a global differential privacy constraint in 
	\cite{gaboardi2016differentially}, \cite{cai2017priv}, \cite{aliakbarpour2018differentially}, \cite{acharya2018differentially} and \cite{canonne2019private}. Further steps into covering other testing problems under global differential privacy have been taken already with works like \cite{aliakbarpour2019private}. %In particular, their novel privatization method maintains sample efficiency of the testing method presented in \cite{diakonikolas2016new}.
	
	Our contributions can be summarized in the following way. Under non-interactive local differential privacy, we provide optimal separation rates for goodness-of-fit testing over Besov balls in the continuous case. To the best of our knowledge, we are the first to provide quantitative guarantees in such a continuous setting. We also provide minimax separation rates for multinomial distributions. In particular, we establish a lower bound \co{when $f_0$ is uniform}, that is completely novel in the definition of the prior distributions leading to optimal rates, and in the way we tackle \co{non-interactive} privacy. Indeed, naive applications of previous information processing inequalities under local privacy lead to suboptimal lower bounds. Finally, we provide an adaptive version of our test, which \co{does not rely on the knowledge of the smoothness parameter $s$ and is} rate-optimal up to a logarithmic factor. So in shorter terms:
	\begin{itemize}
		\item We provide the first minimax lower bound for the problem of goodness-of-fit testing under local privacy constraint over Besov balls. \co{We focus on non-interactive privacy, and $f_0$ being uniform, which directly translates to a lower bound when $f_0$ is nearly uniform}.
		\item We present the first minimax optimal test with the associated \co{non-interactive} local differentially private channel in this continuous setting. \co{The upper bound obtained for this test will hold for any density $f_0 \in \mathbb L_2([0,1])$.}
		\item The test is made adaptive to the smoothness parameter of the unknown density up to a logarithmic term.
		\item A minimax optimal test under \co{non-interactive} privacy can be derived for multinomial distributions as well.
	\end{itemize}
	We start with citing results pertaining to the study of goodness-of-fit testing in the discrete case under local differential privacy. \cite{gaboardi2017local} takes another point of view from ours and provide asymptotic distibutions for a chi-squared statistic applied to noisy observations satisfying the local differential privacy condition. \cite{sheffet2018locally} takes a closer approach to ours and determines a sufficient number of samples for testing between $p=p^0$ and fixed $\sum |p_i-p^0_{i}|$, which has been improved upon by \cite{acharya2018test}. Finally, in parallel with the writing of the present paper, \cite{berrett2020locally} have provided minimax optimal rates of testing for discrete distributions under local privacy, in both $l_1$ and $l_2$ norms. In particular, they tackle both interactive and non-interactive privacy channels and point out a discrepancy in the rates between both cases.
	
	%	 More precisely, they test whether sample points were drawn from a known multinomial distribution. However, we consider continuous densities instead. Besides they work under global differential privacy constraints, whereas we enforce local privacy. Note that they apply Laplace perturbation to the frequencies, whereas we apply the perturbation onto the coefficients of a wavelet basis, and the choice of the basis is crucial in obtaining the optimal rate. Finally and most importantly, they do not provide guarantees on the convergence rates.
	
	Now, the following papers tackle the continuous case. \cite{butucea2019local} provides minimax optimal rates for density estimation over Besov ellipsoids under local differential privacy. Following this paper, we apply Laplace noise to the projection of the observations onto a wavelet basis, although we tackle the different problem of density testing. The difference between density estimation and testing is fundamental and leads in our case to faster rates. A problem closer to density testing is the estimation of the quadratic functional presented in \cite{butucea2020interactive}, where the authors find minimax rates over Besov ellipsoids under local differential privacy. They rely on the proof of the lower bound in the non-interactive case given in a preliminary version of our paper -- see \cite{lam2020minimax}. It was refined in order to improve on the rate in $\alpha$, reaching an optimal rate for low values of $\alpha$. \co{So combining the lower bound we obtain under non-interactive privacy with their study of interactive local privacy, it is possible to note that there is an intrinsic gap in effectiveness that non-interactive privacy cannot hope to close for estimation of the squared functional.}
	
	Finally, the present paper is an iteration over \cite{lam2020minimax}, which only focused on the continuous case. We extend its scope and construct a unified setting to tackle both Besov classes and multinomial distributions, leading to minimax optimal results in both settings.

	The rest of the paper is articulated as follows. In Section~\ref{sec:setting}, we detail our setting and sum up our results. A lower bound on the minimax separation distance for goodness-of-fit testing is introduced in Section~\ref{sec:lowerBound}. Then we introduce a test and a privacy mechanism in Section~\ref{sec:upperBound}. This leads to an upper bound which matches the lower bound. However, in the continuous case, the proposed test depends on a smoothness parameter which is unknown in general. That is the reason why we present a version of the test in Section~\ref{adaptation} that is adaptive to $s$. Afterwards, we conclude the paper with a final discussion in Section~\ref{sec:discussion}. Finally, in the Appendix, the proofs of all the results presented in this paper are contained in Section~\ref{sec:proofs}.
	%	 and discussions on possible alternatives for the proof of the lower bound in Section~\ref{sec:LBnaive}.
	
	All along the paper, $C$ will denote some absolute constant, $c(a,b, \ldots), C(a,b, \ldots) $ will be constants depending only on their arguments. The constants may vary from line to line.

	\section{Setting}
	\label{sec:setting}
	
	\subsection{Local differential privacy}
	
	Let $n$ be some positive integer and $\alpha > 0$. Let $f, f_0$ be densities in $\mathbb L_2([0,1])$ with respect to the Lebesgue measure.
	%	Let $P$ and $P_0$ be distributions with respective densities $f$ and $f_0$ \todo{besoin de densite}according to some common measure $\nu$ \todo{imprecis}.
	Let $X_1,\ldots,X_n$ be i.i.d. random variables with density $f$.
	Equation~\eqref{eq:localPrivacy} defines local differential privacy. However, we define $Z_1,\ldots,Z_n$  satisfying a stronger assumption corresponding to the non-interactive case (see \cite{warner1965randomized} and \cite{evfimievski2003limiting}). It is expressed for all $1 \leq i \leq n$  as
	\begin{equation}\label{eq:noninteractif}
		\sup_{S \in \mathcal Z_i, (x, x') \in [0,1]^2} \frac{Q_i(Z_i \in S|X_i=x)}{Q_i(Z_i \in S|X_i=x')} \leq e^\alpha.
	\end{equation}
	Let $\mathcal{Q}_\alpha$ be the set of joint distributions whose marginals satisfy the condition in Equation~\eqref{eq:noninteractif}\co{, that is, the set of $\alpha$ non-interactive privacy channels associated with $X_1,\ldots,X_n$}.

	\subsection{A unified setting for discrete and continuous distributions}
	\label{sec:unifiedSetting}
	We present a unified setting and end up dealing with densities in $\mathbb L_2([0,1])$ in both the continuous and discrete cases.
	%In the continuous setting, we observe $X_1 \ldots, X_n$ with values in $[0,1]$,  i.i.d. with common density $f$ with respect to the Lebesgue measure on $[0,1]$ and we want to test the null hypothesis $H_0: f=f_0$ against the alternative  $H_1: f\neq f_0$.\\
	In the discrete case, $\widetilde{X_1}, \ldots, \widetilde{X_n}$ are i.i.d. random variables taking their values in $d$ classes denoted by  $\ac{0, 1\ldots,d-1}$ according to the probability vector $p=(p_0, p_1, \ldots, p_{d-1})$.
	For a given probability vector $p^0=(p^0_{0}, p^0_{1}, \ldots, p^0_{d-1})$, we want to test the null hypothesis $H_0:p=p^0$ against the alternative  $H_1:p \neq p^0$.
	In order to have a unified setting, we transform these discrete observations into continuous observations $X_1 \ldots, X_n$ with values in $[0,1]$ by the following process. For all $k \in \ac{0, \ldots, d-1 }$, if we observe $\widetilde{X_i}=k$, 
	we generate $X_i$ by a  uniform distribution on the interval $[k/d, (k+1)/d)$. Note that the variables $X_1 \ldots, X_n$  are i.i.d. with common density $f $ defined for all $x \in [0,1]$ by
	$$ f(x) = \sum_{k=0}^{d-1}  d p_k \un_{ [\frac{k}{d}, \frac{k+1}{d} )}(x). $$
	Similarly, for the probability vector $p^0$, we define the corresponding density $f_0$ for $x \in [0,1]$ by 
	$$ f_0(x) = \sum_{k=0}^{d-1}  d p^0_{k} \un_{ [\frac{k}{d}, \frac{k+1}{d} )}(x). $$
	So\co{, given the definitions of $f$ and $f_0$,} we have the equivalence $ p=p^0 \iff f=f_0$.
	The following equation highlights the connection between the separation rates for densities and for probability vectors. We have
	\begin{equation}
		\label{eq:norm2ContDisc}
		\|f-f_0\|^2_2 = d \sum_{k=0}^{d-1} (p_k-p^0_{k})^2.
	\end{equation}

	\subsection{Separation rates}

	We now define a privacy mechanism and a testing procedure based on the private views $Z_1, \ldots, Z_n$.
	We want to test 
	\begin{equation} \label{def:hypotest}
		H_{0}: f = f_0,~~~~~ \textrm{versus}~~~~~ H_{1}:   f \neq f_0, 
	\end{equation}
	from  $\alpha$-local differentially private views of $X_1,\ldots,X_n$.

	%We assume $X_i\in [0,1]$ in both cases. 
	%\begin{enumerate}
	%\item Multinomial testing: we assume that the observations take their values in $d$ classes denoted by  $\ac{0, 1\ldots,d-1}$. Dividing  the observations by $d$ leads to values in $[0,1]$. The probability vector $p=(p_0, p_1, \ldots, p_{d-1})$ in $\mathbf P$ corresponds to the distribution of the $X_i$'s. Then let $f(x) = \sum_i p_i d \un_{[(i-1)/d,i/d]}(x)$ and we consider $X_i$ generated by density $f$.  We define the corresponding densities $f=\sum_{k=1}^d \beta_k \varphi_k$ and $f_0 = \sum_{k=1}^d \beta^0_k \varphi_k$ using the histogram basis and such that $\int_{i/d}^{(i+1)/d} f(t)dt = p(i)$. Let $X_i$ be distributed according to density $f$.
	%\item Besov testing: $f$ and $f_0$ belong to $\mathbb{L}_2([0,1])$. 
	%\end{enumerate}

	The twist on classical goodness-of-fit testing is in the fact that the samples $(X_1,\ldots,X_n)$ from $f$ are unobserved, we only observe their private views.
	%Then, fixing some $\gamma\in (0,1)$, we say that Problem~\eqref{eq:prob} can be solved, if we can construct a uniformly $\gamma$-consistent test and an $\alpha$-local differentially private channel, that is, if there exist $\varphi$ and $Q$ such that:
	%$$
	%R(\varphi,Q;\varepsilon) = \sup_{f\in H_{0}} \mathbb{P}_{M^n_f}(\varphi(Z_1,\ldots,Z_n)=1) + \sup_{f\in H_{1,\varepsilon}} \mathbb{P}_{M^n_f}(\varphi(Z_1,\ldots,Z_n)=0) \leq \gamma,
	%$$
	%where $M^n_f(\prod_{i=1}^n A_i) = \int \prod Q_i(A_i|x_i) dP_f(x_i)$ is the joint marginal channel.
	%The risk $R(\varphi,Q;\varepsilon)$ corresponds to the probability of the test being wrong in the worst case.
	%
	%
	%Now $\varepsilon \mapsto R(\varphi, Q;\varepsilon)$ is non-increasing, and greater or equal to one when $\varepsilon=0$. Then, we define the separation distance for some $\gamma\in (0,1)$ fixed:
	%$$
	%\varepsilon_\gamma(\varphi,Q) = \sup\{\varepsilon>0 : R(\varphi,Q;\varepsilon) \leq \gamma\}.
	%$$
	%A good test $\varphi$ and a good channel $Q$ are characterized by a small separation distance. So we define the $\alpha$-private minimax separation distance as
	%
	%$$
	%\varepsilon_{\gamma,\alpha}^* = \inf_{\varphi;Q\in\mathcal Q_\alpha} \varepsilon_{\gamma}(\varphi, Q).
	%$$
	For $\alpha > 0$ and $\gamma\in (0,1)$, we construct an $\alpha$-local differentially private channel $Q\in \mathcal{Q}_\alpha$ and a $\gamma$-level test $ \Delta_{\gamma,Q}$  such that
	$$ \mathbb{P}_{Q^n_{f_0}}(\Delta_{\gamma,Q}(Z_1,\ldots,Z_n)=1) \leq \gamma,$$
	where 
	\begin{align*}
		&\mathbb P_{Q^n_{f_0}}((Z_1, \ldots, Z_n) \in \prod_{i=1}^n S_i) =  \int \prod_i Q_i(Z_i \in S_i| X_i=x_i) f_0(x_i)dx_i,
	\end{align*}
	and $Q_i$ is the $i$-th marginal channel of $Q$. 
	
	We then define the uniform separation rate of the test $  \Delta_{\gamma,Q}$ over the  class $\mathcal{C}$ as
	\begin{align}\label{seprateprivate}
		&\rho_n \left(  \Delta_{\gamma,Q},  \mathcal{C} , \beta, f_0 \right) = \inf \Big\{ \rho > 0; \sup_{f \in  \mathcal{C}, \left\| f-f_0 \right\|_2 > \rho } \mathbb{P}_{Q^n_{f}} \left(  \Delta_{\gamma,Q}(Z_1,\ldots, Z_n) = 0 \right) \leq \beta  \Big\}.
	\end{align} 
	%where $ \mathcal{F}_\rho(\mathcal{B}_{s,2, \infty}(R)) = \left\{ f; f-f_0  \in\mathcal{B}_{s,2, \infty}(R), \left\| f-f_0 \right\|_2 > \rho \right\}$.
	A good  channel  $Q$ and a good test $\Delta_{\gamma,Q} $  are characterized by a small uniform separation rate.  This leads us to the definition of the $\alpha$-private minimax separation rate over the class $\mathcal{C} $
	\begin{equation}\label{minimaxprivaterate}
		\rho_n^*\left( \mathcal{C} , \alpha,\gamma, \beta , f_0\right) = \inf_{ Q\in\mathcal{Q_\alpha} } \inf_{\Delta_{\gamma,Q}} \rho_n \left(  \Delta_{\gamma,Q},  \mathcal{C} , \beta, f_0 \right),
	\end{equation}
	where the infimum is taken over all possible $\alpha$-private channels $Q$ and all $\gamma$-level test $ \Delta_{\gamma,Q}$ based on the private observations $Z_1, \ldots, Z_n$. 
	
	Let us now introduce the classes of alternatives $  \mathcal{C} $ over which we will establish $\alpha$-private minimax separation rates. 
	
	\begin{enumerate}
		\item In the discrete case, we define
		\begin{equation}
			\label{eq:discreteClass}
			\mathcal D = \ac{f \in \mathbb L_2([0,1]) ; \exists p = (p_0,\ldots,p_{d-1})\in \mathbb R^d, \sum_{j=0}^{d-1} p_j = 1, f = \sum_{j = 0}^{d-1} p_j \un_{[j/d,(j+1)/d)}},	
		\end{equation}
		which is associated with the class of densities for multinomial distributions over $d$ classes. Then the minimax separation rate of interest will be denoted $\rho_n^*\left( \mathcal{D} , \alpha,\gamma, \beta, f_0 \right)$.
		\item In the continuous case, we consider Besov balls. To define these classes, we consider a pair of compactly supported and bounded wavelets $(\varphi, \psi)$ such that for all $J $ in $\mathbb{N}$, 
		\begin{align*}
			&\ac{2^{J/2} \varphi(2^J (\cdot) -k), k \in \Lambda(J)} \cup \ac{2^{j/2} \psi(2^j (\cdot) -k), j \geq J, k \in \Lambda(j)} 
		\end{align*}
		is an orthonormal basis of $ \mathbb{L}_2([0,1])$.  For the sake of simplicity, we consider the Haar basis where  $\varphi=\un_{[0,1)}$ and $ \psi = \un_{[0,1/2)} - \un_{[1/2,1)} $. 
		In this case, for all $j \in \mathbb{N}$, $ \Lambda(j) = \ac{0, 1, \ldots 2^{j}-1}$.
		
		We denote for all $ j \geq 0, k \in \Lambda(j)$, 
		$$ \alpha_{j,k}(f)=\int 2^{j/2} f \varphi(2^j (\cdot) -k), \text{ and }  \beta_{j,k}(f)=\int 2^{j/2} f \psi(2^j (\cdot) -k).$$
		For $R>0$ and $s>0$, the Besov ball  $\widetilde B_{s,2,\infty}(R)$ with radius $R$ associated with the Haar basis is defined as 
		$$ \widetilde {B}_{s,2,\infty}(R) = \ac{ f \in  \mathbb{L}_2([0,1]), \forall j \geq 0, \sum_{k \in \Lambda(j)} \beta_{jk}^2(f) \leq R^2 2^{-2js}}. $$
		Now note that, if $s < 1$, then there is an equivalence between the definition of $\widetilde {B}_{s,2,\infty}(R)$ and the definition of the corresponding Besov space using moduli of smoothness -- see e.g. Theorem 4.3.2 in \cite{gine2016mathematical}. And for larger $s$, Besov spaces defined with Daubechies wavelets satisfy this equivalence property, as explained in Section 4.3.5 in \cite{gine2016mathematical}.

		We introduce the following class of alternatives: for any $s>0$ and  $R>0,$ we define the set $ \mathcal{B}_{s,2,\infty}(R)$ as follows 
		\begin{equation}\label{BesovLinfty}
			\mathcal{B}_{s,2,\infty}(R) = \ac{ f\in  \mathbb{L}_2([0,1]),  f-f_0 \in \widetilde{B}_{s,2,\infty}(R)}.
		\end{equation} 
		Note that the class $ \mathcal{B}_{s,2,\infty}(R) $ depends on $f_0$ since only the regularity for the difference $f-f_0$ is required to establish the separation rates. Nevertheless, for the sake of simplicity we omit $f_0$ in the notation of this set. The minimax separation rate of interest will be denoted $\rho_n^*\left( \mathcal{B}_{s,2,\infty}(R)  , \alpha,\gamma, \beta, f_0 \right)$.

	\end{enumerate}

	\subsection{Overview of the results}
	\label{sec:overview}
	
	For any $\alpha > 0$, we define $z_\alpha = e^{2\alpha} - e^{-2\alpha} = 2\sinh(2\alpha).$
	
	\paragraph*{Continuous case.}
	The results presented in Theorems~\ref{th:lowerBound} and \ref{bornesup} can be condensed into the following conclusion that holds if $n z_\alpha^2 \geq (\log n)^{1+3/(4s)}$, $s>0$, $R>0,$  $\alpha \geq 1/\sqrt n$, $(\gamma, \beta)  \in (0,1)^2$ such that $ 2\gamma+ \beta <1$,
	\begin{align} \label{mainresult}
		&c\pa{s,R, \gamma, \beta } [(n z_\alpha^2)^{-2s/(4s+3)}  \vee n^{-2s/(4s+1)}]  \nonumber\\
		&
		\quad\quad\quad\quad\quad\quad\leq \rho_n^*\left( \mathcal{B}_{s,2,\infty}(R) , \alpha,\gamma, \beta, \un_{[0,1]} \right) \\
		&\quad\quad\quad\quad\quad\quad\quad\quad\quad\quad\leq C(s,R,\gamma, \beta) \cro{({n} \alpha^2)^{-2s/(4s+3)} \vee n^{-2s/(4s+1)}}\nonumber. 
	\end{align}

	\paragraph*{Comments.}
	\begin{enumerate}
		\item Having $n z_\alpha^2 \geq (\log n)^{1+3/(4s)}$ and $\alpha \geq 1/\sqrt n$ reduces to wanting a sample set large enough, which is a classical non-restrictive assumption.
		\item The upper bound holds for any density $f_0 \in \mathbb L_2([0,1])$ \co{up to $\|f_0\|_2$} and matches the lower bound when $f_0 = \un_{[0,1]}$, as shown in Equation~\eqref{mainresult}. So we can deduce the minimax separation rate for goodness-of-fit testing \co{of Besov densities under a non-interactive privacy constraint, when $f_0$ is nearly uniform.} It can be decomposed into two different regimes, where the rates of our upper and lower bounds match in $n$ as well as in $\alpha$, when $\alpha$ tends to 0. When $\alpha$ is larger than $n^{1/(4s+1)}$, then the minimax rate is of order $n^{-2s/(4s+1)}$, which coincides with the rate obtained in the non-private case in \cite{ingster1987minimax}. The other regime corresponds to $\alpha$ being smaller than $n^{1/(4s+1)}$. The minimax rate is then of order $(n\alpha^2)^{-2s/(4s+3)}$ and so we show a polynomial degradation in the rate due to the privacy constraints. Very similar results have been found for the estimation of the quadratic functional in \cite{butucea2020interactive}. Such a degradation has also been discovered in the problem of second moment estimation and mean estimation, as well as for the density estimation in \cite{butucea2019local}.
		\item Due to having $z_\alpha$ instead of $\alpha$, our upper and lower bounds do not match in $\alpha$ when $\alpha$ is larger than a constant but smaller than $n^{1/(4s+1)}$. This is not an issue in practice, since $\alpha$ will be taken small in order to guarantee privacy.
	\end{enumerate}

	\paragraph*{Discrete case.}
	The results presented in Theorems~\ref{th:lowerBoundDisc} and \ref{bornesupdiscret} can be condensed into the following conclusion that holds if $n \geq (z_\alpha^{-2} d^{3/2}\log d) \vee ((\alpha^2 d^{-1/2}) \wedge d^{1/2})$, $\alpha > 0$, $(\gamma, \beta)  \in (0,1)^2$ such that $ 2\gamma+ \beta <1$, 
	\begin{align} \label{mainresultDisc}
		&c\pa{\gamma, \beta } \cro{((n z_\alpha^2)^{-1/2}d^{1/4})  \vee (n^{-1/2} d^{-1/4})  }\nonumber\\
		&
		\quad\quad\quad\quad\quad\quad\quad\quad\quad\leq \rho_n^*\left( \mathcal{D}, \alpha,\gamma, \beta, \un_{[0,1]} \right)/d^{1/2} \\
		&\quad\quad\quad\quad\quad\quad\quad\quad\quad\quad\quad\quad\quad\quad\quad\quad\leq C(\gamma,\beta) \cro{\pa{(n \alpha^{2})^{-1/2} d^{1/4}} \vee \pa{ n^{-1/2} d^{-1/4}}} \nonumber. 
	\end{align}
	%If there also exists a constant $\kappa$ such that $d \sum_{k=0}^{d-1} p_k^2 \leq \kappa$, then by application 
	% of Corollary~\ref{cor:upperBoundDiscL2}, we have for any $n >0$, $\alpha >0$, $(\gamma, \beta)  \in (0,1)^2$ such that $ 2\gamma+ \beta <1$, 
	%\begin{align} \label{mainresultDisc2}
	%\rho_n^*\left( \mathcal{D}, \alpha,\gamma, \beta, f_0 \right)/d^{1/2} \leq C(\gamma,\beta,\kappa) \cro{\pa{(n \alpha^{2})^{-1/2} d^{1/4}} \vee \pa{ n^{-1/2} d^{-1/4}}}. 
	%\end{align}
	
	\paragraph*{Comments.}
	\begin{enumerate}
		\item Assuming that $n z_\alpha^2 \geq d^{3/2}\log d$ means that the problem gets harder with the dimension, which aligns with the interpretation of the private rate.
		\item We present matching bounds on $\rho_n^*\left( \mathcal{D}, \alpha,\gamma, \beta, \un_{[0,1]} \right)/d^{1/2}$ since it is the usual rate of interest as justified by the combination of Definition~\ref{seprate} and Equation~\eqref{eq:norm2ContDisc}. \co{This helps us conclude on the minimax separation rate for goodness-of-fit testing of discrete distributions under a non-interactive privacy constraint, when $p_0$ is nearly uniform. We present in Theorem~\ref{bornesupdiscret} an upper bound for other definitions of $p_0$ as well.} Here again, we find two regimes corresponding to the classical rate taking over if $\alpha$ is larger than $\sqrt d$. So we can see that the local privacy condition leaves the rate in $n$ unchanged, but the rate in $d$ changes drastically for the testing problem with respect to the $\mathbb L_2$-norm. Indeed, the classical testing problem with $\mathbb L_2$-separation becomes easier as the number of dimension grows, whereas the private rate exhibits the opposite behaviour. 
		% A stronger assumption is the bound on $d \sum_{k=0}^{d-1} p_k^2$ in order to obtain Equation~\eqref{mainresultDisc2}. Indeed it corresponds to $p$ being close to a uniform vector.
		\item Simultaneously and independently of our work, \cite{berrett2020locally} find similar results in the non-interactive case.
	\end{enumerate}
	
	\section{Lower bound}\label{sec:lowerBound}
	
	This section will focus on the presentation of a lower bound on the minimax separation rate defined in Equation~\eqref{minimaxprivaterate} for the problem of goodness-of-fit testing under a \co{non-interactive} differential privacy constraint. The result is presented both in the discrete and the continuous cases, when $f_0$ is the uniform density over $[0,1]$. \co{As seen in \cite{butucea2019local}, an important contribution of our lower bound focusing on the non-interactive case is that, when combined with an interactive privacy channel and a test reaching a smaller separation distance, one can conclude that interactive privacy can lead to better results than what can be achieved under non-interactive privacy in the problem under scrutiny.}
	
	The outline of the lower bound proof relies on a classical scheme, which is recalled below. Nevertheless, the implementation of this scheme in the context of local differential privacy is far from being classical, and we do it in a novel way which leads to a tight lower bound. At the end of the section, a more naive approach will be presented and shown to lead to suboptimal results.
	
	We apply a Bayesian approach, where we will define a prior distribution which corresponds to a mixture of densities such that $\|f-f_0\|_2$ is large enough. Such a starting point has been largely employed for lower bounds in minimax testing, as described in \cite{Baraud2002}. Its application is mainly due to \cite{ingster1993} and inequalities on the total variation distance from \cite{le1986asymptotic}. The result of this approach is summarized in the following lemma.
	
	\begin{lemma}
		\label{lem:lowerBoundTV}
		Let $\mathcal C \subset \mathbb L_2([0,1]).$
		Let $(\gamma, \beta)  \in (0,1)^2$ and $\delta \in [0,1)$ such that $ \gamma+ \beta +\delta <1$. Let $\rho >0$. We define
		$$ \mathcal{F}_{\rho}(\mathcal{C})=\ac{ f  \in  \mathcal{C}, \|f-f_0\|_2 \geq \rho}. $$
		Let $\alpha >0$ and let $Q\in \mathcal{Q_\alpha}$ be some \co{non-interactive} $\alpha$-private channel. 
		Let $\nu_{\rho}$ be some probability measure such that $\nu_{\rho}(\mathcal{F}_{\rho}(\mathcal{C})) \geq 1- \delta$ and let $ Q^n_{\nu_{\rho}}$ be defined, for all measurable set $A$ by
		\begin{align*}
			\Po_{Q^n_{\nu_{\rho}}} &\pa{ (Z_1, \ldots, Z_n) \in A} = \int  \Po_{Q^n_{g}} \pa{ (Z_1, \ldots, Z_n) \in A} d \nu_{\rho}(g).
		\end{align*}
		We note the total variation distance between two probability measures $\Po_1$ and $\Po_2$ as $ \| \Po_1-\Po_2 \|_{TV} = \sup_A |\Po_1(A) - \Po_2(A)| $.
		
		Then  if 
		$$
		\|  \mathbb{P}_{ Q^n_{\nu_{\rho}}} - \mathbb{P}_{ Q^n_{f_0}} \|_{TV} < 1 -  \gamma-\beta - \delta,
		$$
		we have
		$$  \inf_{\Delta_{\gamma,Q}} \rho_n \left(  \Delta_{\gamma,Q}, \mathcal{C}, \beta, f_0 \right) \geq \rho,$$
		where the infimum is taken over all possible $\gamma$-level test, hence satisfying
		$$ \mathbb{P}_{Q^n_{f_0}}(\Delta_{\gamma,Q}(Z_1,\ldots,Z_n)=1) \leq \gamma.$$
	\end{lemma}
	The idea is to establish the connection between the second kind error and the total variation distance between arbitrary distributions with respective supports in $H_0$ and $\mathcal{F}_{\rho}(\mathcal{C})$. It turns out that the closer the distributions from $H_0$ and $\mathcal{F}_{\rho}(\mathcal{C})$ are allowed to be, the higher the potential second kind error. So if we are able to provide distributions from $H_0$ and $\mathcal{F}_{\rho}(\mathcal{C})$ which are close from one another, we can guarantee that the second kind error of any test will be high. The main difficulty lies in finding the right prior distribution $\nu_\rho$ appearing in Lemma~\ref{lem:lowerBoundTV}.
	
	In the discrete case, we obtain the following lower bound.
	\begin{theorem}
		\label{th:lowerBoundDisc}
		Let $(\gamma, \beta)  \in (0,1)^2$ such that $ 2\gamma+ \beta <1$. Let $\alpha >0$. 
		
		We obtain the following lower bound for the $\alpha$-private minimax separation rate defined by Equation~\eqref{minimaxprivaterate}
		for non-interactive channels in $\mathcal{Q_\alpha}$ over the class of alternatives $  \mathcal{D}$ in Equation~\eqref{eq:discreteClass}
		\begin{align*}
			&\rho_n^*\left( \mathcal{D}, \alpha,\gamma, \beta, \un_{[0,1]}\right) / d^{1/2} \geq c\pa{\gamma, \beta} \pa{[(n z_\alpha^2)^{-1/2}d^{1/4} \wedge d^{-1/2}(\log d)^{-1/2}] \vee (n^{-1/2} d^{-1/4})}.
		\end{align*}
	\end{theorem}
	
	\begin{remark}
		In parallel to our work, \cite{berrett2020locally} focus on the case when $\alpha \leq 1$ and find similar results displayed in their Theorem~6.
	\end{remark}
	
	In the continuous case, we obtain the following theorem for Besov balls.
	\begin{theorem}
		\label{th:lowerBound}
		Let $(\gamma, \beta)  \in (0,1)^2$ such that $ 2\gamma+ \beta <1$. Let $\alpha >0, R >0, s>0$. 
		
		We obtain the following lower bound for the $\alpha$-private minimax separation rate defined by Equation~\eqref{minimaxprivaterate}
		for non-interactive channels in $\mathcal{Q_\alpha}$ over the class of alternatives $  \mathcal{B}_{s,2, \infty}(R)$ defined in Equation~\eqref{BesovLinfty}
		\begin{align*}
			&\rho_n^*\left( \mathcal{B}_{s,2,\infty}(R) , \alpha,\gamma, \beta, \un_{[0,1]}\right) \geq c\pa{\gamma, \beta ,R} [[(n z_\alpha^2)^{-2s/(4s+3)} \wedge (\log n)^{-1/2}] \vee n^{-2s/(4s+1)}].
		\end{align*}
	\end{theorem}
	
	\begin{remark}
		These theorems represent a major part of our contributions and lead to the construction of the inequalities presented in Section~\ref{sec:overview}. Note that $(n z_\alpha^2)^{-2s/(4s+3)} \wedge (\log n)^{-1/2}$ reduces to $(n z_\alpha^2)^{-2s/(4s+3)}$ for $n$ large enough and we can reduce the formulation of Theorem~\ref{th:lowerBoundDisc} in the same way with a condition on $n$ being large enough.
	\end{remark}

	\paragraph*{Sketch of proof.} 
	We want to find the largest $\mathbb L_2$-distance between the initial density $f_0$ under the null hypothesis and the density in the alternative hypothesis such that their transformed counterparts by an $\alpha$-private channel $Q$ cannot be discriminated by a test. 
	We will rely on the singular vectors of $Q$ in order to define densities and their private counterparts with ease. Employing bounds on the singular values of $Q$, we define a mixture of densities such that they have a bounded $\mathbb L_2$-distance to $f_0 = \un_{[0,1]}$. We obtain a sufficient condition for the total variation distance between the densities in the private space to be small enough for both hypotheses to be indistinguishable. Then we ensure that the functions that we have defined are indeed densities, and in the continuous case belong to the regularity class $\mathcal{B}_{s,2,\infty}(R)$. Collecting all these elements, the conclusion relies on Lemma~\ref{lem:lowerBoundTV}.
	
	\begin{remark}
		\label{rem:naiveLB}
		%Unfortunately, this natural and simple approach did not allow us to obtain sharp lower bounds (matching with the upper bounds) for the separation rates of goodness-of-fit testing.
		The total variation distance is a good criterion in order to determine whether two distributions are distinguishable. 
		Another natural idea to prove Theorem~\ref{th:lowerBound} is to bound the total variation distance between two private densities by the total variation distance between the densities of the original samples, up to some constants depending on the privacy constraints. Following this intuitive approach, we can provide a lower bound using Theorem 1 in \cite{duchi2013localcomplete} combined with Pinsker's inequality. This approach has been used with success in density estimation in \cite{butucea2019local}. However, the resulting lower bound does not match the upper bound for the separation rates of goodness-of-fit testing presented in our Section~\ref{sec:upperBound}.
		%		 Details on the application of this approach to our setting are provided in Section~\ref{sec:LBnaive} of the appendix.
	\end{remark}
	\section{Definition of a test and privacy mechanism}
	\label{sec:upperBound}
	
	We will firstly define a testing procedure coupled with a privacy mechanism. Their application provides an upper bound on the minimax separation rate for any density $f_0$. The bounds obtained are presented in the right-hand side of Equations~\eqref{mainresult} and \eqref{mainresultDisc} for the continuous and the discrete cases respectively. The test and privacy mechanism will turn out to be minimax optimal since the upper bounds will match the lower bounds obtained in Section~\ref{sec:lowerBound}.
	
	Let us first propose a  transformation of the data, satisfying the differential privacy constraints. 
	\subsection{Privacy mechanism}

	We consider the privacy mechanism introduced in \cite{butucea2019local}. It relies on Laplace noise, which is classical as a privacy mechanism.
	However, applying it to the correct basis with the corresponding scaling is critical in finding optimal results. We denote by $\varphi$ the indicator function on $[0,1)$
	$$ \forall x, \varphi(x) = \un_{[0,1)}(x),$$
	and for all integer $L \geq 1$, we set, for all $k \in \ac{0, \ldots, L-1}$, for all $x \in [0,1)$, 
	$$ \varphi_{L,k}(x) = \sqrt{L} \varphi( Lx-k).$$
	The integer $L$ will be taken as $L= 2^J$ for some $J\geq0$ in the continuous case, and we choose $L=d$ in the discrete case. 
	We define,  for all $i \in \ac{1, \ldots,n}$, the vector $Z_{i,L} = (Z_{i,L,k})_{ k \in \ac{0,\ldots, L-1}}$, by  
	\begin{equation}\label{ChanelCristina}
		\forall k \in \ac{0,\ldots, L-1}, \ Z_{i,L,k}= \varphi_{L,k}(X_i) + \sigma_L W_{i,L,k},
	\end{equation}
	where $( W_{i,L,k})_{1\leq i \leq n, k \in \ac{0,\ldots, L-1}}$ are i.i.d.~Laplace distributed random variables with variance $1$ and
	$$
	\sigma_L=2 \sqrt 2\frac{\sqrt{L}}{\alpha}. 
	$$
	%The definitions of $\varphi_{L,k}$ and $\sigma_d$ depend on the problem considered.
	%\todo{unify notation indicatrice.}
	%
	%\begin{enumerate}
	%\item Multinomial testing: $\varphi_k(x) = \sqrt d\un\{x \in [(k-1)/d, k/d]\}$.
	%\item Besov testing: We consider cases when $d$ can be written as $2^J$. Then
	%	$\varphi_k= 2^{J/2} \un_{[0,1)}(2^J (\cdot) -(k-1))$.
	%\end{enumerate}
	
	\begin{lemma}
		\label{lem:densityPrivate}
		For any $i$, denote $q_{i,L}(\cdot|x)$ the density of the random vector $Z_{i,L}$ with respect to the probability measure $\mu_i$ conditionally to $X_i = x$. Then 
		$$\sup_{S \in \mathcal Z_{i,L}, (x, x')\in[0,1]^2} \frac{Q_i(Z_{i,L} \in S|X_i=x)}{Q_i(Z_{i,L} \in S|X_i=x')} \leq e^\alpha$$
		if and only if there exists $\Omega \in \mathcal Z_{i,L}$ with $\mu_i(Z_{i,L} \in \Omega) = 1$ such that 
		$$\frac{q_{i,L}(z|x)}{q_{i,L}(z|x')} \leq e^\alpha$$
		for any $z \in \Omega$ and any $(x,x') \in [0,1]^2$.
	\end{lemma}
	
	\begin{proof}
		Assume there exists $\Omega$ with $\mu_i(Z_{i,L} \in \Omega) = 1$ such that $\frac{q_{i,L}(z|x)}{q_{i,L}(z|x')} \leq e^\alpha$ for any $z \in \Omega$.
		Let $\tilde S \in \mathcal Z_{i,L}$ and $S = \tilde S \cap \Omega$.
		$$
		\frac{Q_i(Z_{i,L} \in \tilde S|X_i = x)}{Q_i(Z_{i,L} \in \tilde S|X_i = x')} = \frac{Q_i(Z_{i,L} \in S|X_i = x)}{Q_i(Z_{i,L} \in S|X_i = x')}.
		$$
		Then
		\begin{align*}
			\frac{Q_i(Z_{i,L} \in S|X_i = x)}{Q_i(Z_{i,L} \in S|X_i = x')}&= \frac{\int_S q_{i,L}(z|x) d\mu_i(z)}{\int_S q_{i,L}(z|x')d\mu_i(z)} \\
			&\leq \frac{\int_S q_{i,L}(z|x') e^\alpha d\mu_i(z)}{\int_S q_{i,L}(z|x) e^{-\alpha}d\mu_i(z)} = \frac{Q_i(Z_{i,L} \in S|X_i = x')}{Q_i(Z_{i,L} \in S|X_i = x)} e^{2\alpha}.
		\end{align*}
		So
		$$
		\frac{Q_i(Z_{i,L} \in \tilde  S|X_i = x)}{Q_i(Z_{i,L} \in \tilde  S|X_i = x')} \leq e^\alpha.
		$$
		Assume that $Q \in \mathcal{Q}_\alpha$.
		Then for any $S \in \mathcal Z_{i,L}$, we have $Q_i(Z_{i,L} \in S|X_i = x) \leq e^\alpha Q_i(Z_{i,L} \in S|X_i = x').$ That is, for any $S \in \mathcal Z_{i,L}$,
		$$
		\int_S (e^\alpha q_{i,L}(z|x') - q_{i,L}(z|x)) d\mu_i(z) \geq 0.
		$$
		So there exists $\Omega$ with $\mu_i(Z_{i,L} \in \Omega) = 1$ such that $\frac{q_{i,L}(z|x)}{q_{i,L}(z|x')} \leq e^\alpha$ for any $z \in \Omega$.
		
	\end{proof}

	\begin{lemma}
		To each random variable $X_i$ of the sample set $(X_1,\ldots, X_n)$, we associate the vector $Z_{i,L} = (Z_{i,L,k})_{ k \in \ac{0, \ldots, L-1}}$. The random vectors $(Z_{1,L},\ldots,Z_{n,L})$  are  non-interactive  $\alpha$-local differentially private views of  the samples $(X_1,\ldots,X_n)$. Namely, they satisfy the condition in Equation~\eqref{eq:noninteractif}.
	\end{lemma}
	
	The proof in the continuous case can also be found in \cite{butucea2019local} (see Proposition 3.1). 
	We recall here the main arguments for the sake of completeness. 
	\begin{proof}
		The random vectors $(Z_{i,L})_{ 1 \leq i \leq n}$ are i.i.d.~by definition. For any $x_i, x'_i$ in $[0,1]$, for any $z_i \in \mathbb{R}^{L}$, 
		\begin{align*}
			\frac{q_{i,L}(z_i|x_i)}{q_{i,L}(z_i|x'_i)} &= \prod_{k = 0}^{L-1} \exp \Bigg[ \sqrt 2 \frac{\ab{z_{i,k}-\varphi_{L,k}(x'_i)}- \ab{z_{i,k}-\varphi_{L,k}(x_i)} }{\sigma_L} \Bigg]\\
			&\leq \exp\cro{ \sum_{k = 0}^{L-1} \frac{\sqrt 2}{\sigma_L} \pa{ \ab{\varphi_{L,k}(x'_i)} + \ab{\varphi_{L,k}(x_i)}}}.
		\end{align*}
		Since $\varphi_{L,k}(x_i) \neq 0$ for a single value of $ k \in  \ac{0, \ldots, L-1}$, we get
		\begin{eqnarray*}
			\frac{q_{i,L}(z_i|x_i)}{q_{i,L}(z_i|x_i')}  \leq \exp \cro{ \frac{2 \sqrt 2  \| \varphi_{L,k} \|_{\infty}}{\sigma_L}} \leq e^\alpha,
		\end{eqnarray*}
		by definition of $ \sigma_L$, which concludes the proof by application of Lemma~\ref{lem:densityPrivate}. 
	\end{proof}

	\subsection {Definition of the test }
	Let $f_0$ be some fixed density in $\mathbb L_2([0,1])$.
	Our aim is now to define a testing procedure for the testing problem defined in Equation~\eqref{def:hypotest} from the observation of the vectors $(Z_1, \ldots, Z_n)$. Our test statistic $\hat{T}_L$ is defined as 
	\begin{align}\label{stattest}
		\hat{T}_L=\frac1{n(n-1)}\sum_{ i \neq l = 1}^n  \sum_{k = 0}^{L-1} \pa{ Z_{i,L,k} - \alpha^0_{L,k}} \pa{ Z_{l,L,k} - \alpha^0_{L,k}}, 
	\end{align}
	where $ \alpha^0_{L,k} = \int_0^1 \varphi_{L,k}(x) f_0(x) dx$. 
	
	We consider the test function
	\begin{equation}\label{def:test}
		\Delta_{L,\gamma, Q}(Z_1, \ldots, Z_n)= \un_{\hat{T}_L  > t^{0}_L(1-\gamma)}, 
	\end{equation}
	where $ t^{0}_L(1-\gamma)$ denotes the $(1-\gamma)$-quantile of $ \hat{T}_L $ under $H_0$. Note that this quantile can be estimated by simulations, under the hypothesis $ f=f_0$.  We can indeed simulate the vector $(Z_1, \ldots, Z_n)$ if the density of $(X_1, \ldots, X_n)$ is  assumed to be $f_0$. Hence the test rejects the null hypothesis $H_0$ if 
	$$ \hat{T}_L  > t^{0}_L(1-\gamma).  $$
	The test is of level $\gamma$ by definition of the threshold. 
	
	\paragraph*{Comments.} 
	\begin{enumerate}
		\item In a similar way as in \cite{FL06}, the test is based on an estimation of the quantity $\|f - f_0 \|_2^2.$ Note indeed  that 
		$\hat{T}_L$ is an unbiased estimator of $  \| \PJ(f-f_0)  \|_2^2$, where $\PJ$ denotes the orthogonal projection in $\mathbb{L}_2([0,1])$ onto the space generated by the functions $( \varphi_{L,k}, k \in \ac{0,\ldots, L-1})$.   In the discrete case, $f$ and $f_0$ belong to $S_L$ and  $ \PJ(f-f_0)  = f-f_0$. In this case, 
		$$\| \PJ(f-f_0)  \|_2^2 = \| f-f_0  \|_2^2 =   d  \sum_{k = 0}^{d-1} (p_k-p^0_{k})^2.$$

		\item Note that, in the discrete case, we obtain the following expression for the test statistic 
		\begin{align}\label{stattestdiscret}
			\hat{T}_d=\frac d{n(n-1)}  \sum_{k = 0}^{d-1} \sum_{ i \neq l = 1}^n  \pa{ \un_{\widetilde{X_i}=k }-p^0_{k}}\pa{ \un_{\widetilde{X_l}=k} -p^0_{k}}.
		\end{align}
		It is interesting to compare this expression with the $\chi^2$ statistics, which can be written as 
		$$  \sum_{k = 0}^{d-1}  \sum_{ i , l = 1}^n    \frac {\pa{ \un_{\widetilde{X_i}=k }-p^0_{k}}\pa{ \un_{\widetilde{X_l}=k} -p^0_{k}}}{n p^0_{k}}.
		$$ 
		Hence, besides the normalization of each term in the sum by $ p^0_{k}$ in the ${\chi}^2$ test, the main difference lies in the fact that we remove the diagonal terms (corresponding to $i=l$) in our test statistics.
	\end{enumerate}
	
	In the next section, we provide non-asymptotic theoretical results for the power of this test.

	\subsection{Upper bound for the second kind error of the test}
	
	We first provide an upper bound  for the second kind error of our test and privacy channel in a general setting. 
	\begin{theorem} \label{majogene}
		Let $(X_1, \ldots, X_n)$ be i.i.d.~with common density $f$ on $[0,1]$.  Let $f_0$ be some given density on $[0,1]$. We assume that $f$ and $f_0$ belong to $\mathbb{L}_2([0,1])$.  From the  observation of the random vectors $(Z_1,\ldots,Z_n)$ defined by Equation~\eqref{ChanelCristina}, for a given $\alpha >0$, we test the hypotheses 
		\begin{equation*}
			H_{0}: f = f_0,~~~~~ \textrm{versus}~~~~~ H_{1}:   f \neq f_0. 
		\end{equation*}
		We consider the test  $ \Delta_{L,\gamma,Q}$ defined by Equation~\eqref{def:test} with $\hat T_L$ defined in Equation~\eqref{stattest}. 
		The test is obviously of level $\gamma$ by definition of the threshold $t^{0}_L(1-\gamma)$, namely we have
		$$ \mathbb{P}_{Q^n_{f_0}}  \pa{\hat{T}_L \geq   t^{0}_L(1-\gamma)} \leq \gamma.$$
		Under the assumption that
		\begin{equation}\label{cond:majogene}
			\| \PJ(f-f_0)  \|_2^2 \geq \sqrt{ \Var_{Q^n_{f_0}}(\hat{T}_L)/ \gamma } +  \sqrt{\Var_{Q^n_{f}}\pa{\hat{T}_L}/\beta},
		\end{equation}
		the second kind error of the test is controlled by $\beta$, namely we have 
		\begin{equation}\label{errbetaBis}
			\mathbb{P}_{Q^n_f}  \pa{\hat{T}_L \leq  t^{0}_L(1-\gamma)} \leq \beta.
		\end{equation}
		Moreover, we have
		\begin{equation}\label{maj:var}	
			\Var_{Q^n_{f}}\pa{\hat{T}_L} \leq C \cro{\frac{(\sqrt L \|f \|_2+ \sigma_L^2)}{n}  \| \PJ(f-f_0)\|_2^2 + \frac{(\|f\|_2^2 + \sigma_L^4 ) L }{n^2}}.
		\end{equation}
	\end{theorem}
	We give here a sketch of proof of  Theorem~\ref{majogene}. The complete proof of this result is given in Section~\ref{sec:proofUB} of the appendix. Note that it is not fundamentally different from non-private proofs given in \cite{FL06}.
	\paragraph*{Sketch of proof.} 
	We want to establish a condition on $ f-f_0 $, under which the second kind error of the test is controlled by $\beta$.		Denoting by $t_L(\beta)$ the $ \beta$-quantile of $\hat{T}_L $ under $\mathbb{P}_{Q^n_f} $, the condition in Equation~\eqref{errbetaBis} holds as soon as $
	t^{0}_L(1-\gamma) \leq  t_L(\beta)$. Hence, we provide an upper bound for $  t^{0}_L(1-\gamma)$ and a lower bound for $ t_L(\beta)$. 
	By Chebyshev's inequality, we obtain that on the one hand,
	\begin{equation}
		\label{eq:threshBoundBis}
		t^{0}_L(1-\gamma) \leq \sqrt{ \Var_{Q^n_{f_0}}(\hat{T}_L )/ \gamma },
	\end{equation}
	and on the other hand,
	\begin{equation}\label{eq:threshBound2Bis}
		\| \PJ(f-f_0)\|_2^2 - \sqrt{\Var_{Q^n_{f}}\pa{\hat{T}_L }/\beta} \leq t_L(\beta).
	\end{equation}
	We deduce from the inequalities in Equations~\eqref{eq:threshBoundBis} and  \eqref{eq:threshBound2Bis} that Equation~\eqref{errbetaBis} holds as soon as 
	$$ \| \PJ(f-f_0)\|_2^2 \geq   \sqrt{ \Var_{Q^n_{f_0}}(\hat{T}_L )/ \gamma }+  \sqrt{\Var_{Q^n_{f}}\pa{\hat{T}_L}/\beta}.$$
	The main ingredient  to control the variance terms is a control of the variance for  U-statistics of order two which  relies on Hoeffding's decomposition -- see e.g. \cite{Serfling} Lemma A p. 183.  The proof is given in Section~\ref{sec:proofUB}.

	We obtain the following corollary of Theorem \ref{majogene}. It states a result that will be used in order to obtain  an upper bound on the minimax rate both in the discrete and the continuous cases.
	\begin{corollary}\label{cor:uppergene}
		Under the same assumptions as in Theorem \ref{majogene}, we obtain that Equation~\eqref{errbetaBis} holds, that is, the second kind error of the test is controlled by $\beta$ provided that
		\begin{equation}\label{corr:majogene}
			\| \PJ(f-f_0)\|_2^2 \geq  C(\gamma,\beta) \frac{(\|f \|_2 + \|f_0 \|_2 + \sigma_L^2) \sqrt{L}}{n}.
		\end{equation}
		
	\end{corollary}
	In the next sections, we derive from this result upper bounds for the minimax separation rate over Besov balls in the continuous case, and conditions on the $l_2$-distance between $p$ and $p^0$ to obtain a prescribed power for the test in the discrete case. 
	\subsection{Upper bound for the separation distance in the discrete case}
	
	The following theorem provides a sufficient condition on the separation distance between the probability vectors $p$ and $p^0$ for both error kinds of the test to be controlled by $\gamma$ and $\beta$, respectively. This sufficient condition corresponds to an upper bound on the minimax rate $\rho_n^*\left( \mathcal{D}, \alpha,\gamma, \beta, f_0 \right)/d^{1/2}$ in the discrete case.
	
	\begin{theorem} \label{bornesupdiscret}
		Let $p^0=(p^0_{0}, p^0_{1}, \ldots, p^0_{d-1})$ be some given  probability vector. Let $(X_1, \ldots, X_n)$ be i.i.d.~with values in the finite set $\ac{0, 1\ldots,d-1}$ and with common distribution defined by the probability vector $p=(p_0, p_1, \ldots, p_{d-1})$.  
		
		From the  observation of the random vectors $(Z_1,\ldots,Z_n)$ defined by Equation~\eqref{ChanelCristina} for a given $\alpha >0$ with $L=d$, we want to test the hypotheses 
		\begin{equation*}
			H_{0}: p = p^0,~~~~~ \textrm{versus}~~~~~ H_{1}:   p \neq p^0. 
		\end{equation*}
		We consider the test  $ \Delta_{d,\gamma,Q}$ defined by Equation~\eqref{def:test}, which has a first kind error of $\gamma$. 
		The second kind error of the test is controlled by $\beta$, provided that
		$$
		%\label{eq:conddiscreteFull}
		\sqrt{\sum_{i=0}^{d-1} (p_i-p^0_{i})^2}  \geq C(\gamma,\beta) \frac{d^{-1/4}}{n^{1/2}} \pa{d^{1/4} \cro{\pa{\sum_{k=0}^{d-1} (p^0_{k})^2}^{1/4} + n^{-1/2}} + d^{1/2}\alpha^{-1}}.
		$$
		Since $\sum_{k=0}^{d-1} (p^0_{k})^2 \leq 1$, the second kind error of the test is controlled by $\beta$, provided that
		\begin{equation}\label{eq:conddiscrete}
			\sqrt{\sum_{i=0}^{d-1} (p_i-p^0_{i})^2} \geq C(\gamma,\beta) n^{-1/2} (1\vee[d^{1/4}\alpha^{-1}]).
		\end{equation}
	\end{theorem}
	
	\begin{remark}
		Equation~\eqref{eq:conddiscrete} displays a rate that is optimal in $d,n$, when $\alpha$ is smaller than $d^{1/4}$. Besides, the rate in $\alpha$ matches the lower bound asymptotically when $\alpha$ converges to 0.
		The upper bound presented in Theorem~1 from \cite{berrett2020locally} tackles the case when $\alpha$ is smaller than 1 and they find the same rate as ours in their Corollary~2. They present an additional test statistic in order to refine their rates when $p^0$ is not a uniform vector.
	\end{remark}
	
	\begin{corollary}
		\label{cor:upperBoundDiscL2}
		We assume that there exists an absolute constant $\kappa$ such that 
		\begin{equation}
			\label{eq:condBoundL2}
			d \sum_{k=0}^{d-1} (p^0_{k})^2 \leq \kappa.
		\end{equation}
		Then the second kind error of the test is controlled by $\beta$, provided that
		\begin{equation}\label{eq:conddiscreteBorneF}
			\sqrt{\sum_i (p_i-p^0_{i})^2} \geq C(\gamma,\beta,\kappa) \cro{\pa{d^{-1/4}n^{-1/2}} \vee \pa{d^{1/4}n^{-1/2} \alpha^{-1}} \vee n^{-1}}.
		\end{equation}
	\end{corollary}
	
	\begin{remark}
		If we also assume the bound on $d \sum_{k=0}^{d-1} (p^0_{k})^2$ as expressed in Equation~\eqref{eq:condBoundL2}, we find optimal rates in $d,n$ if $n \geq (\alpha^2 d^{-1/2}) \wedge d^{1/2}$. The assumption in Equation~\eqref{eq:condBoundL2} in Lemma~\ref{cor:upperBoundDiscL2} is equivalent to assuming that the function $f_0$ defined in Section~\ref{sec:unifiedSetting} belongs to $\mathbb L^2([0,1])$. It restricts $p^0$ to vectors that are close to being uniform. This coincides with the lower bound on the rate found when $f_0$ is a uniform density.
	\end{remark}

	\subsection{Upper bound for the minimax separation rate over Besov balls}
	
	We provide an upper bound on the uniform separation rate for our test and privacy channel over Besov balls in Theorem~\ref{bornesup}.
	\begin{theorem} \label{bornesup}
		Let $(X_1, \ldots, X_n)$ be i.i.d.~with common density $f$ on $[0,1]$.  Let $f_0$ be some given density on $[0,1]$. We assume that $f$ and $f_0$ belong to $\mathbb{L}_2([0,1])$.
		
		We observe  the random vectors $(Z_1,\ldots,Z_n)$ defined by Equation~\eqref{ChanelCristina} for a given $\alpha >0$ with  the following value for $L$ :
		we assume that $ L=L^*$, where $L^*= 2^{J^*}$, and $J^*$ is the smallest integer $J$ such that $2^J \geq (n \alpha^2)^{2/(4s+3)}\wedge n^{2/(4s+1)}$. 
		
		We want to test the hypotheses 
		\begin{equation*}
			H_{0}: f = f_0,~~~~~ \textrm{versus}~~~~~ H_{1}:   f \neq f_0. 
		\end{equation*}
		%We assume  that $n \alpha^2 \geq 1$. 
		
		We consider the test  $ \Delta_{L^*,\gamma,Q}$ defined by Equation~\eqref{def:test}.
		The uniform separation rate, defined by Equation~\eqref{seprateprivate},  of the test $ \Delta_{L^*,\gamma,Q}$
		over  $  \mathcal{B}_{s,2, \infty}(R)$ defined by Equation~\eqref{BesovLinfty}  
		satisfies  for all $n \in \mathbb{N}^*$, $R >0$, $\alpha \geq 1/\sqrt{n} $, $(\gamma, \beta)  \in (0,1)^2$ such that $ \gamma+ \beta <1$
		\begin{align*}
			&\rho_n \left( \Delta_{L^*,\gamma,Q},  \mathcal{B}_{s,2, \infty}(R)  , \beta, f_0 \right) \leq C( s,R,\|f_0\|_2,\gamma, \beta) \cro{({n} \alpha^2)^{-2s/(4s+3)} \vee n^{-2s/(4s+1)}} .
		\end{align*}
	\end{theorem}
	
	The proof of this result is in Section~\ref{sec:proofUB} of the appendix.
	
	\paragraph*{Comments.} 
	\begin{enumerate}
		\item When the sample set $(X_1, \ldots, X_n)$ is directly observed,  \cite{FL06}  propose a testing procedure with uniform separation rate over the set $  \mathcal{B}_{s,2, \infty}(R)$  controlled by 
		$$ C( s,R, \gamma, \beta) n^{-2s/(4s+1)} ,$$
		which is an optimal result, as proved in \cite{ingster1993}.
		Hence we obtain here a loss in the uniform separation rate, due to the fact that we only observe  $\alpha$-differentially private views  of the original sample. This loss occurs when $ \alpha \leq n^{1/(4s+1)}$. Otherwise, we get the same rate as when the original sample is observed. Comparing this result with the lower bound from Section~\ref{sec:lowerBound}, we conclude that the rate is optimal. 
		\item Finally, having $\alpha < 1/\sqrt n$ represents an extreme case, where the sample size is really low in conjunction with a very strict privacy condition. In such a range of $\alpha$, $J^*$ is taken equal to 0, but this does not lead to optimal rates.
	\end{enumerate}

	The test proposed in Theorem \ref{bornesup} depends on the smoothness parameter $s$ of the Besov ball $ \mathcal{B}_{s,2, \infty}(R) $ via the parameter $J^*$. In a second step, we will propose a test which is adaptive to the smoothness parameter $s$. Namely, in Section \ref{adaptation}, we construct an aggregated testing procedure, which is independent of the smoothness parameter and achieves the minimax separation rates established in Equation~\eqref{mainresult} over a wide range of Besov balls simultaneously, up to a logarithmic term.

	\section{Adaptive tests }\label{adaptation}
	
	In Section~\ref{sec:upperBound}, we have defined in the continuous case a testing procedure which depends on the parameter $J$. The performances of the test depend on this parameter. We have optimized the choice of $J$ to obtain the smallest possible upper bound for the separation rate over the set $\mathcal{B}_{s,2, \infty}(R)$. Nevertheless, the test is not adaptive since this optimal choice of $J$ depends on the smoothness parameter $s$. 
	
	In order to obtain adaptive procedure, we propose, as in \cite{FL06}  to aggregate a collection of tests. For this, we introduce the set 
	$$ \J=\ac{J \in \mathbb{N}, 2^J \leq n^2}$$
	and the aggregated procedure will be based on the collection of test statistics $ ( \hat{T}_{2^J}, J \in \J)$ defined by \eqref{stattest}. 
	
	In Theorem \ref{bornesup}, the 	testing procedure is based on the observation of the  random vectors $(Z_1,\ldots,Z_n)$ defined by Equation~\eqref{ChanelCristina}  with 
	$L= 2^{J^*}$ for the optimized value of $J^*$. Hence, the private views of the original sample depend on the unknown parameter $s$. In order to build the aggregated  procedure, we can no more use the optimized value $J^*$ of $J$ and we need to observe the  random vectors $(Z_1,\ldots,Z_n)$ for all $J \in \J$. In order to guaranty the $\alpha$-local differential privacy, we have to increase slightly the variance of the Laplace perturbation. The privacy mechanism is specified in the following lemma. 
	\begin{lemma} \label{privacy-adapt}	 
		We consider the set $ \J=\ac{J \in \mathbb{N}, 2^J \leq n^2}$. 
		We define, for all $i \in \ac{1, \ldots,n}$,  for all    $J \in \J$,  the vector $\tilde{Z}_{i,2^J} = (\tilde{Z}_{i,2^J,k})_{ k \in \ac{0,\ldots, 2^J-1}}$, by  
		\begin{equation}\label{ChanelCristina-adapt}
			\forall k \in \ac{0,\ldots, 2^J-1}, \ \tilde{Z}_{i,2^J,k}= \varphi_{2^J,k}(X_i) + \tilde \sigma_{2^J} W_{i,2^J,k},
		\end{equation}
		where $( W_{i,2^J,k})_{1\leq i \leq n, k \in \ac{0,\ldots, 2^J-1}}$ are i.i.d.~Laplace distributed random variables with variance $1$ and
		$$
		\tilde \sigma_{2^J}=2 \sqrt 2 |\J| \frac{2^{J/2}}{\alpha}. 
		$$
		For all $1 \leq i \leq n$, we define the random vector $\tilde{Z}_i = (\tilde{Z}_{i,2^J}, J \in \J)$. 
		The random vectors $(\tilde{Z}_i, 1 \leq i \leq n )$ are  non-interactive  $\alpha$-local differentially private views of  the samples $(X_1,\ldots,X_n)$. Namely, they satisfy the condition in Equation~\eqref{eq:noninteractif}.
		
	\end{lemma}
	
	\begin{proof}
		The random vectors $(\tilde Z_{i})_{ 1 \leq i \leq n}$ are i.i.d.~by definition.  Let us denote by $\tilde q_i(\cdot | x_i)$ the density of the vector $\tilde Z_{i}$, conditionally to $X_i=x_i$.   For any $x_i, x'_i$ in $[0,1]$, for any $z_i \in \mathbb{R}^{\sum_{J \in \J} 2^J}$, 
		\begin{align*}
			\frac{\tilde q_i(z_i | x_i)}{\tilde q_i(z_i | x_i')}
			&= \prod_{J \in \J} \prod_{k = 0}^{2^J-1} \exp \Bigg[ \sqrt 2 \frac{\ab{z_{i,k}-\varphi_{2^J,k}(x'_i)}- \ab{z_{i,k}-\varphi_{2^J,k}(x_i)} }{\tilde \sigma_{2^J}} \Bigg]\\
			&\leq \exp\cro{\sum_{J \in \J} \sum_{k = 0}^{2^J-1} \frac{\sqrt 2}{\tilde \sigma_{2^J}} \pa{ \ab{\varphi_{2^J,k}(x'_i)} + \ab{\varphi_{2^J,k}(x_i)}}}.
		\end{align*}
		Since $\varphi_{2^J,k}(x_i) \neq 0$ for a single value of $ k \in  \ac{0, \ldots, 2^J-1}$, we get
		\begin{eqnarray*}
			\frac{\tilde q_i(z_i | x_i)}{\tilde q_i(z_i | x_i')}  &\leq& \exp \cro{2 \sqrt 2 \sum_{J\in \J} \frac{ \| \varphi_{2^J,k} \|_{\infty}}{\tilde \sigma_{2^J}}} \leq e^\alpha,
		\end{eqnarray*}
		by definition of $\tilde \sigma_2^J$, which concludes the proof by application of Lemma~\ref{lem:densityPrivate}. 
	\end{proof}
	Note that $|\J| \leq 1 + 2\log_2(n)$, hence we will have a logarithmic loss for  the separation rates due to the privacy condition for the aggregated procedure. \\
	
	Let us now define the adaptive test. We set, for all $J \in \J$,
	\begin{equation} \label{stattestadapt}
		\widetilde{T}_{J}=\frac1{n(n-1)}\sum_{ i \neq l = 1}^n  \sum_{k = 0}^{2^J-1} \pa{ \tilde Z_{i,2^J,k} - \alpha^0_{2^J,k}} \pa{ \tilde Z_{l,2^J,k} - \alpha^0_{2^J,k}}.
	\end{equation}

	For a given  level $\gamma \in (0,1)$,  the aggregated testing procedure  rejects the hypothesis $H_{0}: f = f_0$ if
	$$ \exists J \in \J, \  \tilde {T}_J  > \tilde t^{0}_J(1-u_\gamma),  $$
	where $u_\gamma$ is defined by
	\begin{equation} \label{ugamma}
		u_\gamma = \sup \Big\{ u \in (0,1), \mathbb{P}_{Q^n_{f_0}} \pa{\sup_{J \in \J} \pa{\tilde {T}_J  - \tilde t^{0}_J(1-u_\gamma)} >0} \leq \gamma \Big\}
	\end{equation}
	and $\tilde t^{0}_J(1-u_\gamma)$ denotes the $ 1-u_\gamma$ quantile of $  \tilde {T}_J$ under $ H_0$. 
	Hence $ u_\gamma$ is the least conservative choice leading to a $\gamma$-level test. We easily notice that $ u_\gamma \geq \gamma/ | \J |$. Indeed, 
	\begin{align*}
		\mathbb{P}_{Q^n_{f_0}} \pa{\sup_{J \in \J} \pa{\widetilde{T}_J  - \tilde t^{0}_J(1-\gamma/ | \J |)} >0}  &\leq  \sum _{J \in \J} \mathbb{P}_{Q^n_{f_0}}   \pa{\widetilde{T}_J  > \tilde t^{0}_J(1-\gamma/ | \J |)} \\
		& \leq  \sum _{J \in \J} \gamma/ | \J | \leq   \gamma.
	\end{align*}
	Let us now consider the second kind error for the aggregated test,  which is the probability to accept the null hypothesis $H_0$, although the alternative hypothesis $H_1$ holds.  This quantity is upper bounded by the smallest second kind error of the tests of the collection, at the price that $ \gamma$ has been replaced by $ u_\gamma$. Indeed, 
	\begin{align}
		\mathbb{P}_{Q^n_{f}} \pa{\sup_{J \in \J} \pa{ \tilde {T}_J  - \tilde t^{0}_J(1-u_{\gamma})} \leq 0} &=  \mathbb{P}_{Q^n_{f}}\pa{\cap_{J \in \J} \pa{\tilde {T}_J  \leq  \tilde t^{0}_J(1-u_\gamma)}} \nonumber \\
		& \leq  \inf_{J \in \J} \mathbb{P}_{Q^n_{f}} \pa{\tilde {T}_J  \leq \tilde t^{0}_J(1-u_\gamma)}. \label{bornetestagreg}
	\end{align}
	We obtain the following theorem for the aggregated procedure. 
	\begin{theorem} \label{bornesupadapt}
		Let $(X_1, \ldots, X_n)$ be i.i.d.~with common density $f$ in $\mathbb{L}_2([0,1])$.  Let $f_0$ be some given density in  $\mathbb{L}_2([0,1])$.  From the  observation of the random vectors  $(\tilde{Z}_i, 1 \leq i \leq n )$ defined in Lemma \ref{privacy-adapt} for a given $\alpha >0$, we want to test the hypotheses 
		\begin{equation*}
			H_{0}: f = f_0,~~~~~ \textrm{versus}~~~~~ H_{1}:   f \neq f_0. 
		\end{equation*}
		We assume that $n\alpha^2/\log^{5/2}(n) \geq 1$. 
		
		We consider the set $ \J= \ac{J \in \mathbb{N}, 2^J \leq n^2}$ and the aggregated test
		$$ \Delta_{\gamma,Q}^{\J}  = \un_{\{ \sup_{J \in \J} \pa{\widetilde{T}_J  - \tilde t^{0}_J(1-u_\gamma)} >0 \}}$$
		where $\widetilde{T}_J$ is defined by Equation~\eqref{stattestadapt}   and $u_\gamma$ by Equation~\eqref{ugamma}. \\
		
		The uniform separation rate, defined by Equation~\eqref{seprateprivate},  of the test $ \Delta_{\gamma,Q}^{\J} $ over the set  $  \mathcal{B}_{s,2, \infty}(R)$ defined by Equation~\eqref{BesovLinfty}  satisfies  for all $n \in \mathbb{N}^*$, $s>0$, $R >0,$ $\alpha > 0$, $(\gamma, \beta)  \in (0,1)^2$ such that $ \gamma+ \beta <1$,
		\begin{align*}
			&\rho_n \left(\Delta_{\gamma,Q}^{\J}  ,  \mathcal{B}_{s,2, \infty}(R)  , \beta, f_0 \right) \\
			&\quad\quad\quad\quad\quad\leq C(\|f_0\|_2, R, \gamma,\beta) \Big[ ({n} \alpha^2/\log^{5/2}(n))^{-2s/(4s+3)}   \vee  (n /\sqrt{\log(n)})^{-2s/(4s+1)} \Big], 
		\end{align*}
	\end{theorem}
	The proof of this result is in Section~\ref{sec:proofAdapt} of the appendix.
	%	\noindent{\bf Comments: } 
	We compare this result with the rates obtained in Theorem~\ref{bornesup}, which has been proved to be optimal. Here, we incur a logarithmic loss due to the adaptation. We recall that in the usual setting, the separation rates obtained by \cite{ingster2000} and \cite{FL06} for adaptive procedures over Besov balls was $ \pa{n/ {\sqrt{\log\log(n)}}}^{-2s/(4s+1)}$. This result was proved to be optimal for adaptive tests in \cite{ingster2000}. In their paper, the log-log term is obtained from exponential inequalities for U-statistics involved in the testing procedure under the null hypothesis. In our setting, obtaining exponential inequalities is not trivial due to the Laplace noise. That is why our logarithmic loss originates from a simple upper bound on the variance of our test statistic under the null. The optimality of the adaptive rates presented in Theorem~\ref{bornesupadapt} remains an open question.

	\section{Discussion}
	\label{sec:discussion}

	Our study of minimax testing rates is in line with Ingster's work and we focus on separation rates in $\mathbb L_2$-norm for goodness-of-fit testing under local differential privacy. We construct a unified setting in order to tackle both discrete and continuous distributions. In the continuous case, we provide the first minimax optimal test and local differentially private channel for the problem of goodness-of-fit testing over Besov balls. This result also holds for multinomial distributions. Besides, in the continuous case, the test and channel remain optimal up to a log factor even if the smoothness parameter is unknown. Among our technical contributions, it is to note that we use a proof technique in the lower bound that could not involve Theorem 1 from \cite{duchi2013localcomplete}. The minimax separation rate turns out to suffer from a polynomial degradation in the private case. However, we point out an elbow effect, where the rate is the same as in the non-private case up to some constant factor if $\alpha$ is large enough.
	% In the discrete case, the upper bound on the separation rate matches the lower bound if $\alpha$ is at most a constant or if $p_0$ is close to a uniform vector \todo{anyways, that is the case in the lower bound...}. 
	Simultaneously and independently, \cite{berrett2020locally} present minimax testing rates for the $l_1$ and $l_2$ norms in the discrete case. We define Besov balls using Haar wavelets, which are equivalent to Besov balls defined using moduli of smoothness when $s < 1$. In order for the equivalence to hold for any $s$, it is possible to define Besov balls using Daubechies wavelets instead. In the proof of our lower bound, we use the disjoint support property of the Haar wavelets, but this can be circumvented taking fewer wavelets in the definition of the prior distributions. A more critical assumption is that $\varphi_{L,k}^2 = \sqrt L \varphi_{L,k}$.
	Future possible works could extend our results to larger Besov classes and study the optimality of the adaptive procedure. \co{Our bounds match when $f_0$ is nearly uniform, and matching bounds for any $f_0$ remain to be proved under non-interactive differential privacy. Finally, finding an interactive privacy channel leading to faster rates than our lower bound would prove that interactive privacy can lead to better results than it is possible under non-interactive privacy for goodness-of-fit testing.}

	%%%%%%%%%%%%%%%%%%%%%%%%%%%%%%%%%%%%%%%%%%%%%%
	%% Single Appendix:                         %%
	%%%%%%%%%%%%%%%%%%%%%%%%%%%%%%%%%%%%%%%%%%%%%%
	%\begin{appendix}
	%\section*{???}%% if no title is needed, leave empty \section*{}.
	%\end{appendix}
	%%%%%%%%%%%%%%%%%%%%%%%%%%%%%%%%%%%%%%%%%%%%%%
	%% Multiple Appendixes:                     %%
	%%%%%%%%%%%%%%%%%%%%%%%%%%%%%%%%%%%%%%%%%%%%%%
	\begin{appendix}

		\section{Proof of the results}
		\label{sec:proofs}
		\subsection{Lower bound: proof of Lemma~\ref{lem:lowerBoundTV}}
		
		Since $\nu_{\rho}(\mathcal{F}_{\rho}(\mathcal{C})) \geq 1- \delta$, we first note that
		\begin{align*}
			\inf_{\Delta_{\gamma,Q}} \sup_{f \in \mathcal{F}_{\rho}(\mathcal C)} \mathbb{P}_{ Q^n_{f}} \left(  \Delta_{\gamma,Q}(Z_1,\ldots, Z_n) = 0 \right)
			&\geq  \inf_{\Delta_{\gamma,Q}}  \mathbb{P}_{ Q^n_{\nu_{\rho}}} \left(  \Delta_{\gamma,Q}(Z_1,\ldots, Z_n) = 0 \right) - \delta\\
			& = \inf_{\Delta_{\gamma,Q}} (\mathbb{P}_{ Q^n_{f_0}} \left(  \Delta_{\gamma,Q}(Z_1,\ldots, Z_n) = 0 \right) \\
			&\qquad\qquad+  \mathbb{P}_{ Q^n_{\nu_{\rho}}} \left(  \Delta_{\gamma,Q}(Z_1,\ldots, Z_n) = 0 \right) \\
			&\qquad\qquad - \mathbb{P}_{ Q^n_{f_0}} \left(  \Delta_{\gamma,Q}(Z_1,\ldots, Z_n) = 0 \right)) - \delta  \\
			&\geq  1-\gamma-  \sup_{\Delta_{\gamma,Q}} \Bigg| \mathbb{P}_{ Q^n_{\nu_{\rho}}} \left(  \Delta_{\gamma,Q}(Z_1,\ldots, Z_n) = 0 \right)\\
			&\quad\quad\quad\quad\quad\quad-
			\mathbb{P}_{ Q^n_{f_0}} \left(  \Delta_{\gamma,Q}(Z_1,\ldots, Z_n) = 0 \right) \Bigg| - \delta
		\end{align*}
		by definition of $\Delta_{\gamma,Q}(Z_1,\ldots, Z_n)$.
		Finally, by definition of the total variation distance,
		\begin{align*}
			&\inf_{\Delta_{\gamma,Q}} \sup_{f \in \mathcal{F}_{\rho}(\mathcal{C})} \mathbb{P}_{ Q^n_{f}} \left(  \Delta_{\gamma,Q}(Z_1,\ldots, Z_n) = 0 \right)  \geq 1-\gamma - \delta - \|  \mathbb{P}_{ Q^n_{\nu_{\rho}}} - \mathbb{P}_{ Q^n_{f_0}} \|_{TV}.
		\end{align*}
		So we have 
		\begin{align*}
			\inf_{\Delta_{\gamma,Q}} \sup_{f \in \mathcal{F}_{\rho}(\mathcal{C})} &\mathbb{P}_{ Q^n_{f}} \left(  \Delta_{\gamma,Q}(Z_1,\ldots, Z_n) = 0 \right) > \beta,
		\end{align*}
		provided that
		$$
		\|  \mathbb{P}_{ Q^n_{\nu_{\rho}}} - \mathbb{P}_{ Q^n_{f_0}} \|_{TV} < 1-\gamma - \beta - \delta . 
		$$
		
		\subsection{Lower bound: proof of Theorem~\ref{th:lowerBound}}
		\label{sec:proofLB}

		An initial version of this proof has been presented in a preprint of \cite{lam2020minimax}, which was then improved upon by \cite{butucea2020interactive} in order to find the matching rate in $\alpha$ and to account for different channels $Q_i$ for each initial observation $X_i$. The proof remains fundamentally the same, however. In line with the rest of the paper, both the discrete and the continuous cases are treated in one unified setting.
		
		In this section, $f_0 = \un_{[0,1]}$.
		
		\subsubsection{Preliminary results}
		
		The following lemma sheds light on the equivalence between the local differential privacy condition and a similar condition on the density of the channel.
		
		\begin{lemma}
			\label{lem:Qdensity}
			Let $Q\in\mathcal{Q_\alpha}$ be an $\alpha$-private channel and $i \leq n$. Let $X_i$ be a random variable with density $f\in\mathbb L_2([0,1])$ with respect to the Lebesgue measure. Then there exists a probability measure with respect to which $Q_i(\cdot|x)$ is absolutely continuous for any $x \in [0,1]$.
		\end{lemma}
		
		\begin{proof}
			Let $\mu_i = \int_{[0,1]} Q_i(\cdot|x) f(x) dx$.
			Let $S \in \mathcal Z_i$ such that $\mu_i(S) = 0$. Then since $Q_i(S|x) \geq 0$ for any $x$, there exists $x$ such that  $Q_i(S|x) = 0$. Now by $\alpha$-local differential privacy, $Q_i(S|x) = 0$ for any $x$.
		\end{proof}
		For the sake of completeness, we prove the following classical inequality between the total variation distance and the chi-squared distance. It will be used in order to reduce the study of the distance between the distributions to that of an expected squared likelihood ratio.
		
		\begin{lemma}
			\label{lem:TVchi2}
			\begin{align*}
				&\|  \mathbb{P}_{ Q^n_{\nu_{\rho}}} - \mathbb{P}_{ Q^n_{f_0}} \|_{TV} \leq \frac12 \pa{ \mathbb{E}_{ Q^n_{f_0}} \cro{L^2_{Q^n_{\nu_{\rho}}}(Z_1, \ldots, Z_n)-1 }}^{1/2},
			\end{align*}
			where
			$
			L_{Q^n_{\nu_{\rho}}}(Z_1, \ldots, Z_n)$ is the likelihood ratio between $Q^n_{\nu_{\rho}}$ and $Q^n_{f_0}$.
		\end{lemma}
		
		\begin{proof}
			We have 
			\begin{align*}
				\|  \mathbb{P}_{ Q^n_{\nu_{\rho}}} - \mathbb{P}_{ Q^n_{f_0}} \|_{TV} = \frac12 \int \ab{ L_{Q^n_{\nu_{\rho}}} -1} d\mathbb{P}_{ Q^n_{f_0}} &=  \frac12  \mathbb{E}_{ Q^n_{f_0}} \cro{\ab{ L_{Q^n_{\nu_{\rho}}}(Z_1, \ldots, Z_n)-1 }}\\
				& \leq  \frac12 \pa{ \mathbb{E}_{ Q^n_{f_0}} \cro{L^2_{Q^n_{\nu_{\rho}}}(Z_1, \ldots, Z_n)-1 }}^{1/2},
			\end{align*}
			by Cauchy-Schwarz inequality and since $ \mathbb{E}_{ Q^n_{f_0}} \pa{ L_{Q^n_{\nu_{\rho}}}(Z_1, \ldots, Z_n)}=1$.
			%We finally have
			%
			%\begin{eqnarray*}
			% \inf_{\Delta_{\gamma,Q}} \sup_{f \in \mathcal{F}_{\rho}(\mathcal{B}_{s,2,\infty}(R, R'))} \mathbb{P}_{ Q_{f}} \left(  \Delta_{\gamma,Q}(Z_1,\ldots, Z_n) = 0 \right) &\geq& 1-\gamma -\delta - \frac12 \pa{ \mathbb{E}_{ Q_{f_0}} \cro{L^2_{Q_{\nu_{\rho}}}(Z_1, \ldots, Z_n)-1 }}^{1/2} \\
			%& > &\beta
			%\end{eqnarray*}
			%provided that
			% $$ \mathbb{E}_{Q_{f_0}} \cro{  L_{Q_{\nu_{\rho}}}^2(Z_1, \ldots, Z_n)} <1 + 4 (1-\gamma-\beta - \delta)^2 . $$
		\end{proof}
		
		The following two lemmas can be interpreted as data processing inequalities. Lemma~\ref{lem:contractionTV} describes the contraction of the total variation distance by a stochastic channel.
		
		\begin{lemma}
			\label{lem:contractionTV}
			Let $\mathbb P_f,\mathbb P_g$ be probability measures over the sample space $[0,1]$ with respective densities $f$ and $g$ with respect to the Lebesgue measure. Let $Q$ be a stochastic channel. Then
			$$
			\| \mathbb P_f - \mathbb P_g \|_{TV} \geq  \| \mathbb P_{Q_f} - \mathbb P_{Q_g} \|_{TV} . 
			$$
			
			%	Let $P$ be a probability measure on a $\sigma$-Algebra $\mathcal{A}$.Then let $X,Y$ be random variables from $(\mathcal A, P)$ to the $\sigma$-Algebra $\tilde A$.
			%	Consider a measurable function $f$. Then 
			%	$$\| P^X - P^Y \|_{TV} \geq  \| P^{f(X)} - P^{f(Y)} \|_{TV} . $$
			
		\end{lemma}
		
		%\begin{proof}
		%%	\begin{align*}
		%%	&\sup_{A \in \mathcal{A}}|P(X \in A) - P(Y \in A) | \\
		%%&\geq  \sup_{\tilde A \in \mathcal{\tilde A}}|P(X \in f^{-1}(\tilde A)) - P(Y \in f^{-1}(\tilde A)) |\\
		%%& = \sup_{\tilde A \in \mathcal{\tilde A}}|P(f(X) \in \tilde A) - P(f(Y) \in \tilde A) |.
		%%	\end{align*}
		%
		%\begin{align*}
		%&\| \mathbb P_{Q_f} - \mathbb P_{Q_g} \|_{TV} = \sup_S \left|\int_{[0,1]} Q_i(S|x) (f(x) - g(x)) d\mu(x)\right|\\
		%& \leq \sup_S \sup_A \left|\int_A Q_i(S|x) (f(x) - g(x)) d\mu(x)\right|\\
		%& \leq \cro{\sup_A \left|\int_A  \sup_S Q_i(S|x) (f(x) - g(x)) d\mu(x)\right|}\\
		%&\quad \vee\cro{\sup_A \left|\int_A  \sup_S Q_i(S|x) (g(x) - f(x)) d\mu(x)\right|}\\
		%&= \cro{\sup_A \left|\int_A (f(x) - g(x)) \un\{f-g \geq 0\}(x) d\mu(x) \right|} \\
		%&\quad\vee \cro{\sup_A \left|\int_A (g(x) - f(x)) \un\{f-g \leq 0\}(x) d\mu(x) \right|} \\
		%&= \sup_A \left|\int_A (f(x) - g(x)) d\mu(x)\right| = \| \mathbb P_f - \mathbb P_g \|_{TV},
		%\end{align*}
		%\todo{reecrire avec beatrice.}
		%since $0\leq Q_i(S|x) \leq 1$ for any measurable set $S$ and point $x$.
		%\end{proof}
		\begin{proof}
			For any measurable set $S$,
			\begin{align*}
				\int_{[0,1]} Q(S|x) (f(x) - g(x)) dx &= \int_{[0,1]} Q(S|x) (f(x) - g(x))\un_{\{f-g \geq 0\}}(x) dx \\
				&+ \int_{[0,1]} Q(S|x) (f(x) - g(x))\un_{\{f-g < 0\}}(x) dx.
			\end{align*}
			Now, since $0\leq Q(S|x) \leq 1$ for any measurable set $S$ and $x\in [0,1]$,
			\begin{align*}
				0 \leq \int_{[0,1]} Q(S|x) (f(x) - g(x))\un_{\{f-g \geq 0\}}(x) dx
				\leq \int_{[0,1]} (f(x) - g(x))\un_{\{f-g > 0\}}(x) dx.
			\end{align*}
			and
			\begin{align*}
				0 \geq \int_{[0,1]} Q(S|x) (f(x) - g(x))\un_{\{f-g < 0\}}(x) dx \geq \int_{[0,1]}  (f(x) - g(x)) \un_{\{f-g < 0\}}(x) dx.
			\end{align*}
			So for any measurable set $S$,
			\begin{align*}
				&\int_{[0,1]}  (f(x) - g(x)) \un_{\{f-g < 0\}}(x) dx\\
				&\quad\quad\quad\quad\quad\quad\quad\quad\quad\quad\leq \int_{[0,1]} Q(S|x) (f(x) - g(x)) dx \\
				&\quad\quad\quad\quad\quad\quad\quad\quad\quad\quad\quad\quad\quad\quad\quad\quad\quad\quad\leq \int_{[0,1]} (f(x) - g(x))\un_{\{f-g > 0\}}(x) dx.
			\end{align*}
			That is, for any measurable set $S$
			\begin{align*}
				\left|\int_{[0,1]} Q(S|x) (f(x) - g(x)) dx\right| &\leq \left|\int_{[0,1]} (f(x) - g(x))\un_{\{f-g > 0\}}(x) dx\right|\\
				& \quad\vee \left|\int_{[0,1]}  (f(x) - g(x)) \un_{\{f-g < 0\}}(x) dx\right|\\
				&= \sup_A \left| \int_A (f(x) - g(x))  dx\right| =\| \mathbb P_f - \mathbb P_g \|_{TV}.
			\end{align*}
		\end{proof}
		
		\subsubsection{Definition of prior distributions}
		\label{sec:defPriorLB}
		
		By Lemma~\ref{lem:Qdensity}, let $Q\in\mathcal{Q_\alpha}$ be a non-interactive $\alpha$-private channel with marginal conditional densities $q_i(z_i|x_i)$ with respect to probability measure $\mu_i$ over the respective sample space $\tilde \Omega_i$ for any $1\leq i \leq n$. In the discrete case, we assume that $p^0$ is a uniform probability vector. By Equation~\eqref{eq:norm2ContDisc}, we can consider the associated uniform density on $[0,1]$. So in both the continuous and the discrete cases, we end up considering a uniform density $f_0$ over $[0,1]$.
		Let $\tilde f_{0,i}(z_i) = \int_0^1 q_i(z_i|x) f_0(x)dx = \int_0^1 q_i(z_i|x)dx$ with the convention $0/0=0$. % \geq e^{-\alpha}$ \todo{by their lemma B3 but we have something similar in our analysis of singular values.
		Let $\mu_i = \int_{[0,1]} Q_i(\cdot|x) f_0(x) dx$ and $K_i: \mathbb L_2([0,1]) \rightarrow \mathbb L_2(\tilde \Omega_i, d\mu_i)$ such that
		$$
		K_i f = \int_0^1 q_i(\cdot | x) f(x) \frac{dx}{\sqrt {\tilde f_{0,i}(\cdot)}}.
		$$
		Let $K_i^*$ denote the adjoint of $K_i$. Then $K_i^*K_i$ is a symmetric integral operator with kernel 
		\begin{equation}
			\label{eq:kerneli}
			F_i(x,y) = \int \frac{q_i(z_i|x) q_i(z_i|y)}{\tilde f_{0,i}(z_i)} d\mu_i(z_i).
		\end{equation}
		And by Fubini's theorem, for any $f \in \mathbb L_2([0,1])$:
		$$
		K_i^* K_i f(\cdot) = \int_0^1 F_i(\cdot,y) f(y) dy.
		$$
		Note that $f_0$ is an eigenfunction of $K_i^* K_i$ associated to the eigenvalue $\lambda_{0,i} = 1$ for all $1 \leq i \leq n.$
		Let
		$$
		K = \sum_{i=1}^n K_i^*K_i / n,
		$$
		which is symmetric and positive semidefinite, and $\lambda_{0} = 1$ is an eigenvalue associated with $f_0$.
		It is an integral operator with kernel
		$$
		F(x,y) = \sum_i F_i(x,y) / n.
		$$	
		We denote by $\psi$ the difference of indicator functions: $ \psi = \un_{[0,1/2)} - \un_{[1/2,1)} $
		and for all integer $L \geq 1$, we set, for all $k \in \ac{0, \ldots, L-1}$, for all $x \in [0,1)$, 
		$$ \psi_{k}(x) = \sqrt{L} \psi(Lx-k).$$
		The integer $L$ will be taken as $L= 2^J$ for some $J\geq0$ in the continuous case, and we choose $L=d/2$ in the discrete case (we assume that $d$ is even).
		We denote by $V$ the linear subspace of $\mathbb{L}^2([0,1])$ generated by the functions $(f_0,\psi_{k}, k \in \ac{0, 1, \ldots L-1})$.
		Then we complete $(f_0)$ into an orthogonal basis $(f_0, u_i)_{1 \leq i \leq L}$ of $V$ with eigenfunctions of $K$ such that $\int u_i(x) dx = 0$ by orthogonality with $f_0$ and $\|u_i\|_2 = 1$. We write the corresponding eigenvalues $\lambda_i$.
		
		Let $z_\alpha = e^{2\alpha} - e^{-2\alpha} \leq 2$ for any $\alpha \in (0,1]$.
		Let $\tilde \lambda_k = (\lambda_k / z_\alpha^2) \vee L^{-1} \geq L^{-1}$.
		Let
		$$
		f_\eta (x) = f_0(x) + \varepsilon \sum_{j=1}^L \eta_j \tilde \lambda_j^{-1/2} u_j(x),
		$$
		where $\eta \in \{-1,1\}^{L}$.
		For all $ i \in \ac{1, \ldots,L}$, $u_i \in \mbox{Span}(\psi_{k}, k \in \{0,1,\ldots,L-1\})$, hence we write
		$$ u_i= \sum_{k = 0}^{L-1} a_{i,k} \psi_{k}.$$	
		Then
		$$
		f_\eta (x) = f_0(x) + \varepsilon \sum_{j=1}^L \sum_{k=0}^{L-1} \eta_j a_{j,k}\tilde \lambda_j^{-1/2}\psi_{k}(x).
		$$
		We define $\nu_\rho$ as the uniform probability measure over $\{f_\eta: \eta \in \{-1,1\}^{L}\}$.
		Now, we can identify the distance between $f_\eta$ and $f_0$. Let $l = \sum_{i=1}^L \un_{\{z_\alpha^{-2} \lambda_i > L^{-1}\}}$. By definition and orthonormality of $(u_i)_{1\leq i \leq L}$, for any $\eta \in \{-1,1\}^{L}$
		\begin{align}
			\label{eq:distff0}
			\|f_\eta - f_0\|_2 &= \varepsilon \sqrt{\sum_{i=1}^{L} \tilde \lambda_i^{-1}\|u_i\|_2^2} = \varepsilon \sqrt{\sum_{i=1}^{L} \tilde \lambda_i^{-1}} \nonumber\\
			&= \varepsilon \sqrt{\sum z_\alpha^2 \lambda_i^{-1} \un_{\{z_\alpha^{-2} \lambda_i > L^{-1}\}} + L \sum \un_{\{z_\alpha^{-2} \lambda_i \leq L^{-1}\}}}\nonumber\\
			&\geq \varepsilon \sqrt{z_\alpha^2 l^2 (\sum_i \lambda_i \un_{\{z_\alpha^{-2} \lambda_i > L^{-1}\}})^{-1}  + L(L- l)},
		\end{align}
		by Cauchy-Schwarz inequality.
		%		the inequality between harmonic and arithmetic means.
		
		So let us provide guarantees on the singular values in order to determine sufficient conditions for $\varepsilon$ to lead to a lower bound on $\rho^*_n \left(  \Delta_{\gamma,Q}, \mathcal{C}, \beta, f_0 \right)$, depending on $\mathcal{C}$.
		
		\subsubsection{Obtaining the inequalities on the eigenvalues}
		
		\begin{lemma}
			\label{lem:ineqEigen}
			Let $K$ be defined as in Section~\ref{sec:defPriorLB} and $(\lambda_i^2)_{0 \leq i \leq L}$ its eigenvalues associated with the orthonormal basis $(f_0, u_i)_{1 \leq i \leq L}$. Then the following inequality holds.
			$$
			\sum_{k=1}^L \lambda_k \leq z_\alpha^2.
			$$
		\end{lemma}
		
		\begin{proof}
			We have
			\begin{align*}
				\sum_{k=1}^{L} \lambda_k &= \sum_{k = 1}^{L} \int_0^1 \int_0^1 \frac{u_k(x) u_k(y)}{n} \sum_{i = 1}^n F_i(x,y)dx dy\\
				&=\frac 1n \sum_{i = 1}^n \int_{\tilde \Omega_i} \sum_{k=1}^{L} \pa{\int_0^1  \frac{q_i(z_i|x)}{\tilde f_{0,i}(z_i)}u_k(x) dx}^2 \tilde f_{0,i}(z_i) d\mu_i(z_i)\\
				&=\frac 1n \sum_{i = 1}^n \int_{\tilde \Omega_i} \sum_{k=1}^{L} \pa{\int_0^1  \pa{\frac{q_i(z_i|x)}{\tilde f_{0,i}(z_i)} - e^{-2\alpha}}u_k(x) dx}^2 \tilde f_{0,i}(z_i) d\mu_i(z_i),
			\end{align*}
			since $\int u_k(x)dx = 0$.
			Now we define $f_{z,i}(x) = \frac{q_i(z_i|x)}{\tilde f_{0,i}(z_i)} - e^{-2\alpha}$ and by Lemma~\ref{lem:densityPrivate},
			\begin{align*}
				0 \leq e^{-\alpha} - e^{-2\alpha} \leq f_{z,i}(x) = \pa{\int_0^1 \frac{q_i(z_i|s)}{q_i(z_i|x)} ds}^{-1} - e^{-2\alpha} \leq e^{\alpha} - e^{-2\alpha} \leq e^{2\alpha} - e^{-2\alpha}.
			\end{align*}
			So $\|f_{z,i}\|_2 \leq e^{2\alpha} - e^{-2\alpha}$.
			Then, by orthonormality of the $u_k$'s, we apply Parseval's inequality:
			\begin{align*}
				\sum_{k=1}^{L} \pa{\int_0^1  \pa{\frac{q_i(z_i|x)}{\tilde f_{0,i}(z_i)} - e^{-2\alpha}}u_k(x) dx}^2 &= \sum_{k=1}^{L} \langle {f_{z,i}, u_k} \rangle^2 = \|\sum_{k=1}^L \langle f_{z,i}, u_k \rangle u_k\|_2^2\leq \|f_{z,i}\|_2^2 \leq z_\alpha^2.
			\end{align*}
			Finally, $
			\int_{\tilde \Omega_i} \tilde f_{0,i}(z_i) d\mu_i(z_i) = 1
			$ leads to $\sum \lambda_k \leq z_\alpha^2$.
			
		\end{proof}
		Then from Equation~\eqref{eq:distff0} and by application of Lemma~\ref{lem:ineqEigen},
		\begin{equation}
			\label{eq:fL2Epsilon}
			\|f_\eta - f_0\|_2 \geq L \varepsilon \sqrt{(L^{-1} l)^2 + 1 - (L^{-1} l)} \geq L \varepsilon \sqrt{3/4}.
		\end{equation}
		So for the discrete case, by Equation~\eqref{eq:norm2ContDisc},
		\begin{equation}
			\label{eq:pL2Epsilon}
			\sqrt{\sum_{i=0}^{d-1} (p_{i} -  p^0_{i})^2} \geq \sqrt d \varepsilon \sqrt{3/4}.
		\end{equation}

		\subsubsection{Information bound}
		
		Let $\varepsilon > 0$,  for all $\eta \in \ac{-1,1}^{L}$, we define
		$$
		\tilde f_{\eta,i}(z) = \tilde f_{0,i}(z) + \varepsilon \sum_{j=1}^{L}  \eta_j \tilde \lambda_j^{-1/2} \int_0^1 q_i(z_i|x) u_j(x) dx.
		$$		
		We consider the expected squared likelihood ratio:
		\begin{align*}
			\mathbb{E}_{Q^n_{f_0}} \cro{  L_{Q^n_{\nu_{\rho}}}^2(Z_1, \ldots, Z_n)}
			&= \mathbb{E}_{Q^n_{f_0}} \E_{\eta,\eta'}  \prod_{i=1}^n \left(1 + \varepsilon \frac{\sum_{j=1}^{L} \tilde \lambda_j^{-1/2} \eta_j \int_0^1 q_i(Z_i|x) u_j(x) dx}{\tilde f_0(Z_i)}\right)\\
			&\quad\quad\quad\quad\quad\quad\left(1 + \varepsilon \frac{\sum_{j=1}^{L} \tilde \lambda_j^{-1/2} \eta'_j \int_0^1 q_i(Z_i|x) u_j(x) dx}{\tilde f_0(Z_i)}\right) \\
			&= \mathbb{E}_{Q^n_{f_0}}  \E_{\eta,\eta'}  \prod_{i=1}^n \Bigg(1 + \frac{ \varepsilon \sum_{j=1}^{L} \tilde \lambda_j^{-1/2} \eta_j \int_0^1 q_i(Z_i|x) u_j(x) dx}{\tilde f_0(Z_i)} \\
			&\quad\quad\quad\quad\quad\quad+ \frac{ \varepsilon \sum_{j=1}^{L} \tilde \lambda_j^{-1/2} \eta'_j \int_0^1 q_i(Z_i|x) u_j(x) dx}{\tilde f_0(Z_i)} \\
			& + \frac{\varepsilon^2 \sum_{j,l=1}^{L} \tilde \lambda_j^{-1/2} \tilde \lambda_l^{-1/2} \eta_j\eta'_l \int_0^1 q_i(Z_i|x) u_j(x) dx \int_0^1 q_i(Z_i|y) u_l(y) dy}{\tilde f_0(Z_i)^2} \Bigg).
		\end{align*}
		Now, for any $j$, 
		\begin{align*}
			\E_{Q_{f_0}} \cro{\frac{\int_0^1 q_i(Z_i|x) u_j(x) dx}{\tilde f_0(Z_i)}} 		=  \int_0^1 \int_{\tilde \Omega_i} q_i(z|x) d\mu_i(z) u_j(x) dx 
			=\int_0^1 u_j(x) dx= 0,
		\end{align*}
		by orthogonality with uniform vector $f_0$.
		
		And, by Equation~\eqref{eq:kerneli},
		\begin{align*}
			\E_{Q_{f_0}} \cro{\frac{\int_0^1 q_i(Z_i|x) u_j(x) dx \int_0^1 q_i(Z_i|y) u_l(y) dy}{\tilde f_0(Z_i)^2}} = \int_0^1 \int_0^1 F_i(x,y) u_j(x) u_l(y) dx dy.
		\end{align*}
		So since $1+u \leq \exp u$ for any $u$,
		\begin{align*}
			\mathbb{E}_{Q^n_{f_0}} &\cro{  L_{Q^n_{\nu_{\rho}}}^2(Z_1, \ldots, Z_n)}\\
			&\leq \E_{\eta,\eta'} \exp \pa{\varepsilon^2 \sum_{j,l=1}^{L} \tilde \lambda_j^{-1/2} \tilde \lambda_l^{-1/2} \eta_j\eta'_l  n \int_0^1   \int_0^1 F(x,y) u_j(x) u_l(y) dx dy}.
		\end{align*}
		Now 
		\begin{align*}
			\int_0^1   \int_0^1 F(x,y) u_j(x)  u_l(y) dx dy &= \lambda_j \int_0^1 u_j(x) u_l(x)  dx = \lambda_j \un_{\{j=l\}},
		\end{align*}
		since $u_j$ is an eigenfunction of $K$ and by orthonormality.
		So 
		\begin{align*}
			\mathbb{E}_{Q^n_{f_0}} \cro{  L_{Q^n_{\nu_{\rho}}}^2(Z_1, \ldots, Z_n)} \leq \E_{\eta,\eta'} \exp \pa{n \varepsilon^2 \sum_{j=1}^{L} \tilde \lambda_j^{-1} \eta_j\eta'_j   \lambda_j} \leq \E_{\eta,\eta'} \exp \pa{n \varepsilon^2 \sum_{j=1}^{L}  \eta_j\eta'_j   z_\alpha^2}.
		\end{align*}
		Then 
		\begin{align*}
			\mathbb{E}_{Q^n_{f_0}}\left[ L_{Q^n_{\nu_{\rho}}}^2(Z_1, \ldots, Z_n)\right] \leq \prod_{j=1}^{L} \cosh(n\varepsilon^2 z_\alpha^2) \leq \prod_{j=1}^{L} \exp(n^2 \varepsilon^4 z_\alpha^4) \leq \exp(n^2\varepsilon^4 z_\alpha^4 L). 
		\end{align*}
		Then, in order to apply Lemma~\ref{lem:lowerBoundTV} combined with \ref{lem:TVchi2}, let us find a sufficient condition for
		$$ \mathbb{E}_{Q^n_{f_0}} \cro{  L_{Q^n_{\nu_{\rho}}}^2(Z_1, \ldots, Z_n)} < 1 + 4 (1-\gamma-\beta-\gamma)^2. $$
		So let us choose $\varepsilon$ and $J$ in order to ensure that
		$$ 
		\exp(n^2\varepsilon^4 z_\alpha^4 L) < 1+4(1-2\gamma - \beta)^2,
		$$
		i.e.
		$$
		L \varepsilon^4 \leq (n z_\alpha^2)^{-2} \log\left[1+4(1-2\gamma-\beta)^2 \right],
		$$
		i.e.
		\begin{equation}
			\label{eq:rhoSuff}
			\varepsilon \leq (n z_\alpha^2)^{-1/2} \left(\frac{\log\left[1+4(1-2\gamma-\beta)^2 \right]}{L}\right)^{1/4}.
		\end{equation}

		\subsubsection{Sufficient condition for $f_\eta$ to be non-negative}
		
		\begin{lemma}
			\label{lem:suffCondFDisc}
			If 
			$$
			\varepsilon  \leq \frac{L^{-1}}{\sqrt{2\log(2L/\gamma)}},
			$$
			then there exists $A_\gamma \subset \{-1,1\}^{L}$ such that $\Po_{\nu_\rho}(\eta \in A_\gamma) \geq 1-\gamma$ and for any $\eta \in A_\gamma$, $f_\eta$ is a density.
		\end{lemma}
		\begin{proof}
			Let
			$$
			A_\gamma = \{\eta: |\sum_j \eta_j a_{j,k} \tilde \lambda_j^{-1/2}| \leq \sqrt{2 L \log(2 L / \gamma)}\}.
			$$
			Since for all $i$, $u_i$ is orthogonal to $f_0$, uniform density on $[0,1]$,  we have that $ \int_0^1 f_\eta(x) dx =1$ and we just have to prove that $f_\eta$ is nonnegative. 
			We remind the reader that
			$$ u_i= \sum_{k =0}^{L-1} a_{i,k} \psi_{k}.$$
			The bases $(u_1, \ldots, u_{L})$ and $(\psi_{k}, k \in \{0,1,\ldots,L-1\})$ are orthonormal. This implies that the matrix $A=(  a_{i,k})_{1\leq i \leq L, k \in \{0,1,\ldots,L-1\}}$ is orthogonal. So
			$$ \forall i, \sum_{k =0}^{L-1} a_{i,k}^2=1, \ \forall k, \sum_{i =1}^{L} a_{i,k}^2=1.$$
			Hence we have for all $x \in [0,1]$, 
			\begin{align} \label{expansion}
				(f_\eta - f_0)(x) =  \sum_{k =0}^{L-1} \sum_{i =1}^{L}  \eta_i \varepsilon \tilde \lambda_i^{-1/2} a_{i,k}\psi_{k}(x).
			\end{align}
			The functions $ (\psi_{k},k \in \{0,1,\ldots,L-1\})$ have disjoint supports and  $\sup_{x \in [0,1]} | \psi_{k}(x)| = L^{1/2}$. Hence $f_\eta$ is nonnegative if and only if for any $k\in \Lambda(J)$
			\begin{equation} \label{Conddensity}
				L^{1/2} \ab{ \sum_{i =1}^{L}  \eta_i \varepsilon \tilde \lambda_i^{-1/2} a_{i,k}} \leq 1.
			\end{equation}
			By definition of $\nu_\rho$,  we have that
			$f_\eta$ is a density with probability larger than $1-\gamma$ under the prior $\nu_{\rho}$
			as soon as Equation~\eqref{Conddensity}  holds with probability larger than $1-\gamma$.
			That is,
			$$ \mathbb{P}_{\nu_{\rho}}\pa{\forall  k \in \{0,1,\ldots,L-1\}, L^{1/2} \ab{ \sum_{i =1}^{L}  \eta_i \varepsilon \tilde \lambda_i^{-1/2} a_{i,k} } \leq 1} \geq 1-\gamma,$$ where $( \eta_1, \ldots, \eta_L)$ are i.i.d.~Rademacher random variables. Using Hoeffding's inequality, we get for all $x>0$, for all $k \in \{0,1,\ldots,L-1\}$, 
			\begin{align*}
				&\mathbb{P}_{\nu_{\rho}}\pa{\ab{ \sum_{i =1}^{L}  \eta_i \varepsilon \tilde \lambda_i^{-1/2} a_{i,k} } >x}  \leq  2 \exp \pa{ \frac{-2x^2}{\sum_{i =1}^{L}  
						(2 \varepsilon \tilde \lambda_i^{-1/2}a_{ik})^2}}.
			\end{align*}
			Hence
			\begin{align*}
				&\mathbb{P}_{\nu_{\rho}}\pa{ \exists  k \in \{0,1,\ldots,L-1\}, \ab{ \sum_{i = 1}^{L}  \eta_i \varepsilon \tilde \lambda_i^{-1/2} a_{i,k} } >x} \leq  2L \exp \pa{ \frac{-x^2}{2\sum_{i =1}^{L}  
						( \varepsilon \tilde \lambda_i^{-1/2}a_{ik})^2  }}.
			\end{align*}
			So the probability of having the existence of some $k \in \{0,1,\ldots,L-1\}$ such that
			\begin{align*}
				&\ab{ \sum_{i =1}^{L}  \eta_i \varepsilon \tilde  \lambda_i^{-1/2} a_{i,k} }  > \sqrt{ 2 \sum_{i =1}^{L}   (\varepsilon \tilde \lambda_i^{-1/2}a_{ik})^2 \log(2 L/\gamma)}
			\end{align*}
			is smaller than  $\gamma$.
			Hence, $f_\eta$ is a density with probability larger than $1-\gamma$ under the prior $\nu_{\rho}$ as soon as for any $k\in \Lambda(J)$, $$2L\sum_{i =1}^{L}   ( \varepsilon \tilde \lambda_i^{-1/2}a_{ik})^2  \log(2L/\gamma) \leq 1.
			$$
			Now by definition, $\tilde \lambda_i^{-1/2} \leq L^{1/2}.$ So we have the sufficient condition
			$$ 2L \varepsilon^2 L \log(2L/\gamma) \sum_{i = 1}^{L}  a_{ik}^2  \leq 1.
			$$
			And $\sum_{i =1}^{L}  a_{ik}^2 = 1$ leads to the following sufficient condition,
			$$
			\varepsilon  \leq \frac{L^{-1}}{\sqrt{2\log(2L/\gamma)}}.
			$$
		\end{proof}
		
		\subsubsection{Sufficient conditions for $f_\eta\in \mathcal{F}_{\rho}(\mathcal{B}_{s,2,\infty}(R))$, only in the continuous case}
		
		We first prove the following points.
		\begin{lemma} 
			\label{lem:suffCondFNEw}
			If 
			$$
			\varepsilon  \leq \frac{L^{-1} (1 \wedge R L^{-s})}{\sqrt{2\log(2L/\gamma)}},
			$$
			then there exists $A_\gamma \subset \{-1,1\}^{d}$ such that $\Po_{\nu_\rho}(\eta \in A_\gamma) \geq 1-\gamma$ and for any $\eta \in A_\gamma$,
			\begin{enumerate}[a)]
				\item $f_\eta$ is a density.
				\item $ f_\eta \in \mathcal{B}_{s,2,\infty}(R)$.
			\end{enumerate}
		\end{lemma}
		\begin{proof}
			We consider the same event as in the proof of Lemma~\ref{lem:suffCondFDisc}:
			$$
			A_\gamma = \{\eta: |\sum_j \eta_j a_{j,k} \tilde \lambda_j^{-1/2}| \leq \sqrt{2L \log(2L / \gamma)}\}.
			$$

			\begin{enumerate}[a)]
				\item In the same way as in Lemma~\ref{lem:suffCondFDisc}, $f_\eta$ is a density.				
				\item 
				%We deduce from Equation~\eqref{expansion} that  $ \forall j \neq J$, $ \forall k$, $ \langle f_\eta -f_0, \psi_{j,k} \rangle =0$  since $ 
				%\langle  \psi_{j,l} , \psi_{J,k} \rangle =0$ for all $k$ and $l$. 
				For all $ k \in \{0,1,\ldots,L-1\}$, 
				$$ \langle f_\eta -f_0, \psi_{k} \rangle  = \varepsilon \sum_{i =1}^{L}  \eta_i \tilde \lambda_i^{-1/2} a_{i,k} .$$
				Hence $ f_\eta \in \mathcal{B}_{s,2,\infty}(R)$ if and only if 
				$$  \sum_{k = 0}^{L-1} \varepsilon^2 \pa{ \sum_{i =1}^{L}  \eta_i \tilde \lambda_i^{-1/2} a_{i,k}}^2 \leq R^2 L^{-2s}.$$
				But we also have that for any $\eta \in A_\gamma,$
				$$  \sum_{k = 0}^{L-1} \varepsilon^2 \pa{ \sum_{i =1}^{L}  \eta_i \tilde \lambda_i^{-1/2} a_{i,k}}^2 \leq \varepsilon^2 L^{2} 2 \log(2L / \gamma).$$
				So $ f_\eta \in \mathcal{B}_{s,2,\infty}(R)$ if
				$$
				\varepsilon \leq R L^{-(s+1)}/\sqrt{2 \log(2L / \gamma)}.
				$$
			\end{enumerate}
			
			\subsubsection{Conclusion}
			
			\begin{enumerate}
				\item Discrete case.
				
				So combining Equation~\eqref{eq:rhoSuff} and Lemma~\ref{lem:suffCondFDisc}, we obtain the following sufficient condition in order to apply Lemma~\ref{lem:lowerBoundTV}:
				\begin{align*}
					\varepsilon \leq &\cro{(n z_\alpha^2)^{-1/2} \left(\frac{\log\left[1+4(1-2\gamma-\beta)^2 \right]}{L}\right)^{1/4}} \wedge \frac{L^{-1}}{\sqrt{2\log(2L/\gamma)}}.
				\end{align*}
				So, by Equation~\eqref{eq:pL2Epsilon}, if
				\begin{align*}
					\sqrt{\sum_{k=0}^{d-1} (p_k-p^0_{k})^2} \leq \sqrt{3/4} \Bigg(&\cro{(n z_\alpha^2)^{-1/2} d^{1/4} \left(\log\left[1+4(1-2\gamma-\beta)^2 \right]\right)^{1/4}} \\
					&\wedge \frac{d^{-1/2}}{\sqrt{2\log(2d/\gamma)}} \Bigg),
				\end{align*}
				then we can define densities $f_\eta$ such that the errors are larger than $\gamma$ and $\beta$.
				So
				\begin{align*}
					&\inf_{\Delta_{\gamma,Q}} \rho_n \left(  \Delta_{\gamma,Q}, \mathcal{D}, \beta, f_0 \right)/d^{1/2} \\
					&\geq c\pa{\gamma, \beta} [((n z_\alpha^2)^{-1/2}d^{1/4}) \wedge (d\log d)^{-1/2}].
				\end{align*}
				
				Now, we also have
				\begin{align*}
					&\rho_n^*\left( \mathcal{D}, \alpha,\gamma, \beta, f_0 \right)  \geq \rho_n^*\left( \mathcal{D}, +\infty,\gamma, \beta, f_0 \right),
				\end{align*}
				where $\rho_n^*\left( \mathcal{D}, +\infty,\gamma, \beta, f_0 \right)$ corresponds to the case where there is no local differential privacy condition on $Q$. In particular, taking $Q$ such that $Z=X$ with probability 1 reduces the private problem to the classical testing problem. Now, the data processing inequality in Lemma~\ref{lem:contractionTV} justifies that such a $Q$ is optimal by contraction of the total variation distance.
				%\inf_{ Q\in\mathcal{Q_\alpha} } \inf_{\Delta_{\gamma,Q}} \rho_n \left(  \Delta_{\gamma,Q},  \mathcal{C}_s , \beta \right),
				And the classical result leads to having 
				$\rho_n^*\left( \mathcal{C}, +\infty,\gamma, \beta, f_0 \right) = c\pa{\gamma, \beta} n^{-1/2} d^{-1/4}.$
				
				So, we have 
				\begin{align*}
					&\rho_n^*\left( \mathcal{D}, \alpha,\gamma, \beta, f_0\right) / d^{1/2} \geq c\pa{\gamma, \beta} [((n z_\alpha^2)^{-1/2}d^{1/4}) \wedge (d\log d)^{-1/2}] \vee (n^{-1/2} d^{-1/4}).
				\end{align*}

				\item Continuous case.
				
				So combining Equation~\eqref{eq:rhoSuff} and Lemma~\ref{lem:suffCondFNEw}, we obtain the following sufficient condition in order to apply Lemma~\ref{lem:lowerBoundTV}:
				\begin{align*}
					\varepsilon \leq &\cro{(n z_\alpha^2)^{-1/2} \left(\frac{\log\left[1+4(1-2\gamma-\beta)^2 \right]}{L}\right)^{1/4}} \wedge \frac{L^{-1} (1 \wedge R L^{-s})}{\sqrt{2\log(2L/\gamma)}}.
				\end{align*}
				So,  by Equation~\eqref{eq:fL2Epsilon}, if
				\begin{align*}
					\|f-f_0\|_2 \leq \sqrt{3/4} \Bigg(&\cro{(n z_\alpha^2)^{-1/2} L^{3/4} \log^{1/4} \left(1+4(1-2\gamma-\beta)^2 \right)} \wedge \frac{(1 \wedge R L^{-s})}{\sqrt{2\log(2L/\gamma)}} \Bigg),
				\end{align*}
				then, taking $J$ as the largest integer such that $2^J \leq c\pa{\gamma, \beta, R}  (nz_\alpha^2)^{2/(4s+3)}$, we obtain:
				\begin{align*}
					&\inf_{\Delta_{\gamma,Q}} \rho_n \left(  \Delta_{\gamma,Q}, \mathcal{B}_{s,2,\infty}(R), \beta, f_0 \right) \geq c\pa{\gamma, \beta,R} (n z_\alpha^2)^{-2s/(4s+3)} (\log n)^{-1/2}.
				\end{align*}
				%\todo{say cannot do better than observing directly X (case where take away privacy condition, so bigger space and data processing inequality leads to contraction of total variation distance)}
				Now, we also have
				\begin{align*}
					&\rho_n^*\left( \mathcal{B}_{s,2,\infty}(R), \alpha,\gamma, \beta, f_0 \right)  \geq \rho_n^*\left( \mathcal{B}_{s,2,\infty}(R) , +\infty,\gamma, \beta, f_0 \right),
				\end{align*}
				where $\rho_n^*\left( \mathcal{B}_{s,2,\infty}(R) , +\infty,\gamma, \beta, f_0 \right)$ corresponds to the case where there is no local differential privacy condition on $Q$. In particular, taking $Q$ such that $Z=X$ with probability 1 reduces the private problem to the classical testing problem. Now, the data processing inequality in Lemma~\ref{lem:contractionTV} justifies that such a $Q$ is optimal by contraction of the total variation distance.
				%\inf_{ Q\in\mathcal{Q_\alpha} } \inf_{\Delta_{\gamma,Q}} \rho_n \left(  \Delta_{\gamma,Q},  \mathcal{C}_s , \beta \right),
				And the classical result leads to having 
				$\rho_n^*\left( \mathcal{B}_{s,2,\infty}(R) , +\infty,\gamma, \beta, f_0 \right) = c\pa{\gamma, \beta,R} n^{-2s/(4s+1)}.$
				
				So, we have 
				\begin{align*}
					&\rho_n^*\left( \mathcal{B}_{s,2,\infty}(R) , \alpha,\gamma, \beta, f_0\right)  \geq c\pa{\gamma, \beta ,R} [(n z_\alpha^2)^{-2s/(4s+3)} (\log n)^{-1/2} \vee n^{-2s/(4s+1)}].
				\end{align*}
				%		\begin{eqnarray*}
				%			\sum_{k \in \Lambda(J)} \pa{ \sum_{i =2}^{K}  \eta_i \epsilon_i \lambda_i^{-1} a_{i,k}}^2 &=& \| f_\eta-f_0 \|^2 \\
				%			&=& {\sum_{i=2}^{K}  \epsilon_i^2 \lambda_i^{-2}} \\
				%			&\leq &  \epsilon^2 2^{2J}
				%		\end{eqnarray*}
				%		by definition of the $\epsilon_i's$ and since $K-1\leq 2^J$.
				%		Hence we obtain that Equation~\eqref{CondBesov} is a  sufficient condition to ensure that $ f_\eta-f_0 \in \mathcal{B}_{s,2,\infty}(R)$. 
			\end{enumerate}

		\end{proof}

		\subsection {Proof of the upper bound}\label{sec:proofUB}
		
		In this section, $f_0$ is some fixed density in $\mathbb L_2([0,1])$.

		\subsubsection{ Proof of Theorem \ref{majogene}}
		
		We prove the bound on the variance term $\Var_{Q^n_{f}}\pa{\hat{T}_L}$ given in Equation~\eqref{maj:var}. 
		%We control this term for a general density $f$, the case $f=f_0$ beeing a particular case.
		Let us define
		\begin{align*}
			&\hat{U}_L  = \frac1{n(n-1)}\sum_{ i \neq l = 1}^n  \sum_{k=0}^{L-1} \pa{ Z_{i,L,k} - \alpha_{L,k}} \pa{ Z_{l,L,k} - \alpha_{L,k}},
		\end{align*}
		$$ \hat{V}_L =  2   \sum_{k=0}^{L-1} \pa{ \alpha_{L,k} -  \alpha_{L,k}^0} \frac1n \sum_{i=1}^n ({ Z_{i,L,k}} - \alpha_{L,k}),$$
		where $ \alpha_{L,k} = \int_0^1 \varphi_{L,k}(x) f(x) dx$ and $ \alpha^0_{L,k} = \int_0^1 \varphi_{L,k}(x) f_0(x) dx$.
		Then we obtain the Hoeffding's decomposition of the U-statistic $\hat{T}_L$, namely
		$$ \hat{T}_L =  \hat{U}_L +  \hat{V}_L  + \| \PJ(f-f_0)\|_2^2.$$
		We first control the variance of the degenerate U-statistic $\hat{U}_L$ which can be written as 
		$$ \hat{U}_L  =  \frac1{n(n-1)}\sum_{ i \neq l = 1}^n  h_L(Z_{i,L}, Z_{l,L}), $$
		where 
		$$   h_L(Z_{i,L}, Z_{l,L}) = \sum_{k =0}^{L-1} \pa{ Z_{i,L,k} - \alpha_{L,k}} \pa{ Z_{l,L,k} - \alpha_{L,k}}.$$
		In order to provide an upper bound for the variance $\Var_{Q^n_{f}} (\hat{U}_L )$, let us first state a  lemma controlling the variance of a $U$-statistic of order $2$. This result is a particular case of  Lemma 8 in \cite{HSICtest}. 
		
		\begin{lemma}	\label{lem:VarU}
			Let $h$ be a symmetric function with $2 $ inputs, $Z_1, \ldots ,Z_n$ be independent and identically distributed random vectors and $U_n$ be the $U$-statistic of order $2$ defined by 
			\begin{equation*}
				U_n = \frac{1}{n(n-1)} \sum_{i \neq l =1}^n h(Z_i,Z_l).
			\end{equation*}
			The following inequality gives an upper bound on the variance of $U_n$, 
			\begin{equation*}
				\Var(U_n) \leq C \left( \frac{\sigma^2}{n} +  \frac{s^2}{n^2} \right),
			\end{equation*}
			where $\sigma^2 = \Var \left( \E[h (Z_1,Z_2) \mid Z_1] \right)$ and $s^2 = \Var \left( h (Z_1, Z_2) \right)$.  
			\label{Var_Ustat}
		\end{lemma}
		We have that  $ \E_{Q^n_{f}}[h_L (Z_1,Z_2) \mid Z_1]  =0$, hence the  first term in the upper bound of the variance vanishes. In order to bound the term $s^2$, we write
		\begin{align*}
			h_L (Z_1,Z_2) &= \sum_{k =0}^{L-1}  \pa{  \varphi_{L,k}(X_1) - \alpha_{L,k}}\pa{  \varphi_{L,k}(X_2) - \alpha_{L,k}}  + \sigma_L^2 \sum_{k =0}^{L-1}  W_{1,L,k}W_{2,L,k}\\
			&\quad + \sigma_L  \sum_{k =0}^{L-1} W_{1,L,k} \pa{  \varphi_{L,k}(X_2) - \alpha_{L,k}}  +  \sigma_L  \sum_{k =0}^{L-1}  W_{2,L,k} \pa{  \varphi_{L,k}(X_1) - \alpha_{L,k}}.
		\end{align*}
		So, since $\E_{Q^n_{f}} \pa{  \varphi_{L,k}(X_i) - \alpha_{L,k}} = 0$ and $\E(W_{i,L,k} )= 0$ for any $i$. Using independence properties, we therefore have
		\begin{align*}
			&\Var_{Q^n_{f}} \left( h_L (Z_1, Z_2) \right) \\
			& =  \Var_{Q^n_{f}} \Bigg[   \sum_{k =0}^{L-1}  \pa{  \varphi_{L,k}(X_1) - \alpha_{L,k}} \pa{  \varphi_{L,k}(X_2) - \alpha_{L,k}} \Bigg] + \Var_{Q^n_{f}} \cro{ \sigma_L^2  \sum_{k =0}^{L-1} W_{1,L,k}W_{2,L,k}}\\
			&\quad+  2 \Var_{Q^n_{f}} \cro{ \sigma_L  \sum_{k =0}^{L-1} W_{1,L,k} \pa{  \varphi_{L,k}(X_2) - \alpha_{L,k}}}.
		\end{align*}
		Now, by independence of $X_1$ and $X_2$,
		\begin{align*}
			&\Var_{Q^n_{f}} \cro{  \sum_{k =0}^{L-1}  \pa{  \varphi_{L,k}(X_1) - \alpha_{L,k}}\pa{  \varphi_{L,k}(X_2) - \alpha_{L,k}} } \\
			&=   \sum_{k,k' =0}^{L-1}  \E\cro{\pa{  \varphi_{L,k}(X_1) - \alpha_{L,k}} \pa{  \varphi_{L,k'}(X_1) - \alpha_{L,k'}} } \\
			&\quad\quad\quad\quad\quad\quad \E\cro{\pa{  \varphi_{L,k}(X_2) - \alpha_{L,k}} \pa{  \varphi_{L,k'}(X_2) - \alpha_{L,k'}} } \\
			&=   \sum_{k ,k'=0}^{L-1} \cro{ \int \varphi_{L,k} \varphi_{L,k'} f -  \alpha_{L,k}\alpha_{L,k'}}^2.
		\end{align*}
		So
		\begin{align*}
			&\Var_{Q^n_{f}} \cro{  \sum_{k =0}^{L-1}  \pa{  \varphi_{L,k}(X_1) - \alpha_{L,k}}\pa{  \varphi_{L,k}(X_2) - \alpha_{L,k}} } \\
			&= \int \int \pa{  \sum_{k =0}^{L-1}\varphi_{L,k} (x) \varphi_{L,k} (y)}^2 f(x) f(y) dx dy   -2 \int \pa{    \sum_{k =0}^{L-1} \alpha_{L,k} \varphi_{L,k} (x)}^2  f(x) dx \\
			&\quad\quad+ \pa{  \sum_{k =0}^{L-1} \alpha_{L,k}^2 }^2.
			%&\leq \| f \|_{\infty}^2  \int \int     \sum_{k,l =0}^{L-1} \varphi_{L,k} (x) \varphi_{L,k} (y)  \varphi_{L,l} (x) \varphi_{L,l} (y) dx dy  +  \pa{\int f^2}^2\\
			%&\leq C(\|f\|_{\infty} )L
		\end{align*}
		In order to control the first term of the variance, note that by definition of the functions $\varphi_{L,k}$, we have that, for all $x \in [0,1]$,  $\varphi_{L,k}(x) \varphi_{L,k'}(x)=0$ if $k \neq k'$, and that $\varphi_{L,k}^2 = \sqrt L \varphi_{L,k}$.  Hence, 
		\begin{align*}
			\int \int \pa{  \sum_{k =0}^{L-1}\varphi_{L,k} (x) \varphi_{L,k} (y)}^2 f(x) f(y) dx dy &= L  \sum_{k =0}^{L-1} \alpha_{L,k}^2 \\
			& \leq L \|f\|_2^2.
		\end{align*}
		%\todo{ $ = L \| \PJ(f)\|_2^2$ }
		Since the second term of the variance is nonpositive,  and the third term is controlled by $ \|f\|_2^4$,%\todo{$\| \PJ(f)\|_2^4$}
		we obtain
		$$\Var_{Q^n_{f}} \cro{  \sum_{k =0}^{L-1}  \pa{  \varphi_{L,k}(X_1) - \alpha_{L,k}}\pa{  \varphi_{L,k}(X_2) - \alpha_{L,k}} } 
		\leq  L \|f\|_2^2 + \|f\|_2^4 \leq 2 L \|f\|_2^2. $$
		%\todo{$\leq 2L \| \PJ(f)\|_2^2$ since $\| \PJ(f)\|_2^2 \leq L$.}
		By independence of the variables $ (W_{i,L,k})$, 
		\begin{align*}
			&\Var \pa{ \sigma_L^2  \sum_{k =0}^{L-1} W_{1,L,k}W_{2,L,k}} = \sigma_L^4  \sum_{k =0}^{L-1}\Var(  W_{1,L,k}W_{2,L,k})  = L \sigma_L^4.
		\end{align*}
		Finally, using again the  independence of the variables $ (W_{1,L,k})_{k \in \ac{0, \ldots, L-1}}$,  and their independence with $X_2$, 
		\begin{align*}
			&\Var_{Q^n_{f}} \cro{ \sigma_L  \sum_{k =0}^{L-1}  W_{1,L,k} \pa{  \varphi_{L,k}(X_2) - \alpha_{L,k}}} \\
			&= 
			\sigma_L^2 \E_{Q^n_{f}}  \cro{ \sum_{k,k' =0}^{L-1}W_{1,L,k} W_{1,L,k'} 
				\pa{  \varphi_{L,k}(X_2) - \alpha_{L,k}}  \pa{  \varphi_{L,k'}(X_2) - \alpha_{L,k'}} }\\
			&= \sigma_L^2  \sum_{k =0}^{L-1}  \E(W_{1,L,k}^2)  \E_{Q^n_{f}}\cro{\pa{  \varphi_{L,k}(X_2) - \alpha_{L,k}}^2} \\
			&\leq     \sigma_L^2 \sum_{k =0}^{L-1} \int   \varphi_{L,k}^2 f \leq	  \sigma_L^2 L \sum_{k =0}^{L-1} \int_{k/L}^{(k+1)/L} f  \leq	  \sigma_L^2 L
		\end{align*}
		since $ \int_0^1 f =1$.
		This leads to the following upper bound for  $  \Var_{Q^n_{f}} \left( h_L (Z_1, Z_2) \right)$, 
		$$  \Var_{Q^n_{f}} \left( h_L (Z_1, Z_2) \right) \leq (2\|f\|_2^2 + \sigma_L^2 +  \sigma_L^4 ) L , $$
		% \todo{$\leq (2 \| \PJ(f)\|_2^2 + \sigma_L^2 +  \sigma_L^4 ) L $}
		from which, by application of Lemma~\ref{lem:VarU}, we deduce that
		$$  \Var_{Q^n_{f}} \pa{ \hat{U}_L}  \leq 2 \frac{(\|f\|_2^2 + \sigma_L^4 ) L }{n^2}.$$
		% \todo{$\leq 2 \frac{(\| \PJ(f)\|_2^2 + \sigma_L^4 ) L }{n^2}$}
		% Then, plugging the last inequality into Equation~\eqref{eq:threshBound}, we have
		% \begin{equation}\label{maj.quantile}
		% t^{0}_J(1-\gamma) \leq  C(\|f_0\|_{\infty} )(1 + \sigma_J^2) \frac{K^{1/2}}{n\sqrt{\gamma}}. 
		% \end{equation}
		Let us now compute $ \Var_{Q^n_{f}} (\hat{V}_L)$. Since $\hat{V}_L$ is centered,
		\begin{align*}
			&\Var_{Q^n_{f}}\pa{\hat{V}_L } = \E_{Q^n_{f}}(\hat{V}_L^2)\\
			&\quad\quad= \frac4{n^2} \sum_{k,k' =0}^{L-1} \pa{ \alpha_{L,k} -  \alpha_{L,k}^0}  \pa{ \alpha_{L,k'} -  \alpha_{L,k'}^0}    \sum_{i,l=1}^n  \E_{Q^n_{f}} \cro{ ({ Z_{i,L,k}} - \alpha_{L,k}) ({ Z_{l,L,k'}} - \alpha_{L,k'})}.
		\end{align*}
		Note that, if $i \neq l$, 
		$$ \E_{Q^n_{f}} \cro{ ({ Z_{i,L,k}} - \alpha_{L,k}) ({ Z_{l,L,k'}} - \alpha_{L,k'})} =0.$$
		Moreover,
		\begin{align*}
			&\E_{Q^n_{f}} \cro{ ({ Z_{i,L,k}} - \alpha_{L,k}) ({ Z_{i,L,k'}} - \alpha_{L,k'})}\\
			& \quad\quad\quad\quad = \E \big[ ({ \varphi_{L,k}}(X_i) - \alpha_{L,k}) 
			({  \varphi_{L,k'}}(X_i)- \alpha_{L,k'})  + \sigma_L^2\E_{Q^n_{f}}( W_{i,L,k} W_{i,L,k'}) \big] \\
			& \quad\quad\quad\quad = \int  \varphi_{L,k}  \varphi_{L,k'} f -  \alpha_{L,k}  \alpha_{L,k'} + 2  \sigma_L^2 \un_{k=k'}.
		\end{align*}
		Hence,
		\begin{align*}
			&\Var_{Q^n_{f}}\pa{\hat{V}_L} \\
			&=  \frac4{n} \sum_{k,k' =0}^{L-1}  \pa{ \alpha_{L,k} -  \alpha_{L,k}^0}  \pa{ \alpha_{L,k'} -  \alpha_{L,k'}^0}  \pa{
				\int \varphi_{L,k}  \varphi_{L,k'} f -  \alpha_{L,k}  \alpha_{L,k'} + 2  \sigma_L^2 \un_{k=k'}} \\
			&=  \frac4{n} \int \pa{\sum_{k =0}^{L-1}(  \alpha_{L,k} -  \alpha_{L,k}^0)  \varphi_{L,k}}^2 f  -  \frac4{n} \pa{\sum_{k =0}^{L-1} \alpha_{L,k}( \alpha_{L,k} -  \alpha_{L,k}^0)}^2 \\
			&\quad\quad + \frac8{n}\sigma_L^2 \sum_{k =0}^{L-1} (  \alpha_{L,k} -  \alpha_{L,k}^0)^2\\
			&\leq \frac4{n} \sum_{k =0}^{L-1} ( \alpha_{L,k} -  \alpha_{L,k}^0)^2  \int \varphi_{L,k}^2 f +  \frac8{n}\sigma_L^2 \sum_{k =0}^{L-1}(  \alpha_{L,k} -  \alpha_{L,k}^0)^2\\
			&\leq \frac1n\pa{4 \sqrt{L} \|f \|_2  +  8\sigma_L^2 }   \sum_{k =0}^{L-1} ( \alpha_{L,k} -  \alpha_{L,k}^0)^2 
		\end{align*}
		%{\color{blue}
		%\begin{align*}
		% &\leq \frac4{n} \sqrt{L}  \pa{\sum_{k =0}^{L-1} ( \alpha_{L,k} -  \alpha_{L,k}^0)^2  \alpha^0_{L,k} + \sum_{k =0}^{L-1} ( \alpha_{L,k} -  \alpha_{L,k}^0)^3}  +  \frac8{n}\sigma_L^2 \sum_{k =0}^{L-1}(  \alpha_{L,k} -  \alpha_{L,k}^0)^2\\
		% &\leq \frac1n\pa{4 \sqrt{L} \|f_0 \|_2  +  8\sigma_L^2 }   \sum_{k =0}^{L-1} ( \alpha_{L,k} -  \alpha_{L,k}^0)^2 + \frac4{n} \sqrt{L} \sum_{k =0}^{L-1} ( \alpha_{L,k} -  \alpha_{L,k}^0)^3\\
		% &\leq \frac1n\pa{4 \sqrt{L} \pa{\|f_0 \|_2 + \pa{\sum_{k =0}^{L-1} ( \alpha_{L,k} -  \alpha_{L,k}^0)^2}^{1/2} }   +  8\sigma_L^2 }   \sum_{k =0}^{L-1} ( \alpha_{L,k} -  \alpha_{L,k}^0)^2
		%\end{align*}
		%by norm inequality.
		%}
		since by Cauchy Schwarz inequality,
		$$  0 \leq \int \varphi_{L,k}^2 f  = \sqrt{L}  \int \varphi_{L,k} f \leq  \sqrt{L}  \| \varphi_{L,k} \|_2  \| f\|_2  =  \sqrt{L}    \| f\|_2  .$$
		We finally obtain, 
		$$ \Var_{Q^n_{f}}\pa{\hat{V}_L } \leq C\frac{(\sqrt L \|f \|_2+ \sigma_L^2)}{n}  \| \PJ(f-f_0)\|_2^2. $$
		%\todo{$\leq C\frac{(\sqrt L \|f_0 \|_2 + \sqrt L \| \PJ(f-f_0)\|_2 + \sigma_L^2)}{n}  \| \PJ(f-f_0)\|_2^2. $}
		Collecting the upper bounds for  $\Var_{Q^n_{f}}\pa{\hat{U}_L } $ and for $\Var_{Q^n_{f}}\pa{\hat{V}_L } $, we obtain the inequality from Equation~\eqref{maj:var}, that we remind here:
		\begin{align*}
			\Var_{Q^n_{f}}\pa{\hat{T}_L} \leq C \cro{\frac{(\sqrt L \|f \|_2+ \sigma_L^2)}{n}  \| \PJ(f-f_0)\|_2^2 + \frac{(\|f\|_2^2 + \sigma_L^4 ) L }{n^2}}.
		\end{align*}
		%\todo{$\leq C\cro{\frac{(\sqrt L \|f_0 \|_2 + \sqrt L \| \PJ(f-f_0)\|_2 + \sigma_L^2 + L/n)}{n}  \| \PJ(f-f_0)\|_2^2+ \frac{(\| f_0\|_2^2 + \sigma_L^4 ) L }{n^2}}. $}
		%By Chebyshev's inequality, 
		%$$ \mathbb{P}_{Q^n_f}\pa{ \hat{T}_L \leq \E_{Q^n_{f}}( \hat{T}_L) - \sqrt{{Var_{f}\pa{\hat{T}_L } }/{\beta}}} \leq \beta.$$
		%Now $\E_{Q^n_{f}}( \hat{T}_L) = \| \PJ(f-f_0)\|^2$, hence we obtain the following lower bound on the $\beta$-quantile,
		%\begin{equation*}
		%\| \PJ(f-f_0)\|^2 - \sqrt{Var_{f}\pa{\hat{T}_L }/\beta} \leq t_L(\beta).
		%\end{equation*}
		%Moreover, 
		%\begin{align*}
		%	 &\sqrt{Var_{f}\pa{\hat{T}_L }/\beta}  \leq  C(\|f\|_{\infty},\beta)  \big[ \| \PJ(f-f_0)\| \pa{  \frac{ 1+\sigma_L}{\sqrt{n}}} +  \frac{ (1+\sigma_L^2) \sqrt{L}}{n} \big] \\
		%	& \leq  \frac12  \| \PJ(f-f_0)\|^2 +  C(\|f\|_{\infty},\beta)  \frac{ (1+\sigma_L^2)  \sqrt{L}}{n},
		%\end{align*}
		%where we have used the inequality $ ab \leq a^2/2 + b^2/2$ . This leads us to the following lower bound for $ t_J(\beta)$
		%\begin{equation}
		%\label{eq:betaLB}
		% t_L(\beta) \geq  \frac12  \| \PJ(f-f_0)\|^2 -  C(\|f\|_{\infty},\beta)  \frac{ (\sigma_L^2+1) \sqrt{L}}{n}.
		%\end{equation}
		%

		\subsubsection{ Proof of  Corollary \ref{cor:uppergene}}
		From Equation~\eqref{maj:var} and taking $f=f_0$, we obtain
		\begin{align*}
			\sqrt{\Var_{Q^n_{f_0}}  \pa{\hat{T}_L}/\gamma } \leq C(\gamma) \frac{(\|f_0\|_2 + \sigma_L^2 ) \sqrt L}{n}.
		\end{align*}
		%\todo{$\leq C(\gamma) \frac{(\|f_0\|_2 + \sigma_L^2 ) \sqrt L}{n}$}
		Moreover, we deduce from \eqref{maj:var} that 
		$$\sqrt{\Var_{Q^n_{f}}  \pa{\hat{T}_L}/\beta }\leq C(\beta) \cro{ \frac{(L^{1/4} \|f \|_2^{1/2}+\sigma_L)}{\sqrt{n}} \| \PJ(f-f_0)  \|_2  + \frac{(\|f \|_2+\sigma_L^2 ) \sqrt{L}}{n}}.$$
		%\todo{$\leq C(\beta) \cro{ \frac{(L^{1/4} \|f_0 \|_2^{1/2} + L^{1/4} \| \PJ(f-f_0)\|_2^{1/2} +\sigma_L + \sqrt{L/n})}{\sqrt{n}} \| \PJ(f-f_0)  \|_2  + \frac{(\|f_0 \|_2+\sigma_L^2 ) \sqrt{L}}{n}}$}
		Using the inequality between geometric and harmonic means, we get
		$$ \sqrt{\Var_{Q^n_{f}}  \pa{\hat{T}_L}/\beta }\leq   \frac12  \| \PJ(f-f_0)\|_2^2 +  C(\beta)  \frac{(\|f \|_2+\sigma_L^2 ) \sqrt{L}}{n}.$$
		%\todo{$\leq \frac14  \| \PJ(f-f_0)\|_2^2 +  C(\beta)  \frac{(\|f_0 \|_2 + \| \PJ(f-f_0)\|_2 +\sigma_L^2 + \sqrt{L}/n ) \sqrt{L}}{n}$}
		%\todo{$\leq \frac12  \| \PJ(f-f_0)\|_2^2 +  C(\beta)  \frac{(\|f_0 \|_2 +\sigma_L^2 + \sqrt{L}/n ) \sqrt{L}}{n}$}
		%
		%\todo{and $\sigma_L^2 = 8 L / \alpha^2$ dominates if $\alpha \leq C L^{1/4} \sqrt n $. In the discrete case, the term f0 always dominates.}
		We conclude the proof by using the condition in Equation~\eqref{cond:majogene}.
		
		\subsubsection{Proof of  Theorem \ref{bornesupdiscret}}
		
		We recall that we have defined 
		$$ f =  d \sum_{k =0}^{d-1} p_k \un_{[k/d, (k+1)/d)}, \  f_0 =  d \sum_{k =0}^{d-1} p^0_{k} \un_{[k/d, (k+1)/d)} . $$
		We obtain from  Corollary \ref{cor:uppergene} that the second kind error of the test is controlled by $\beta$ if 
		\begin{align*}
			\frac{3}{2}\| \PJ(f-f_0)\|_2^2 &\geq  C(\gamma,\beta) \frac{(\|f \|_2 + \|f_0 \|_2 + \sigma_L^2) \sqrt{L}}{n}.
		\end{align*}
		%\todo{$\| \PJ(f-f_0)\|_2^2 \geq  C(\gamma,\beta)  \frac{(\|f_0 \|_2 +\sigma_L^2 + \sqrt{L}/n ) \sqrt{L}}{n}.$ }
		In the discrete case, by definition, $f$ and $f_0$ belong to $S_L$, hence $\| \PJ(f-f_0)\|_2 = \| f-f_0\|_2$ and $\|f\|_2 \leq \|f_0\|_2 + \| f-f_0\|_2$. So we have the following sufficient condition.
		$$
		\| f-f_0\|_2^2 \geq C(\gamma,\beta) \frac{(\|f_0\|_2 + \sigma_L^2 + \sqrt L / n)\sqrt L}{n}
		$$
		Moreover, we have
		$$ \| f-f_0\|_2^2=d   \sum_{k =0}^{d-1} (p_k-p^0_{k})^2.$$
		We recall that $L=d$ and $\sigma_L = 2\sqrt{2d}/\alpha$. That is, the sufficient condition turns out to be
		$$d   \sum_{k =0}^{d-1} (p_k-p^0_{k})^2  \geq C(\gamma,\beta) \frac{d^{1/2}}{n} \pa{\|f_0 \|_2 + d\alpha^{-2} + d^{1/2}n^{-1}}.
		$$
		By definition of $f_0$, we have that
		$$ \|f\|_2^2  = d\sum_{k=0 }^{d-1}  (p_{0,k})^2.$$
		Finally,  we obtain the following condition
		$$
		\sqrt{\sum_{i=0}^{d-1} (p_i-p^0_{i})^2}   \geq C(\gamma,\beta) \frac{d^{-1/4}}{n^{1/2}} \pa{\cro{d \sum_{k=0}^{d-1} (p^0_{k})^2}^{1/4} + d^{1/2}\alpha^{-1} + d^{1/4}n^{-1/2}}.
		$$

		\subsubsection{Proof of  Theorem \ref{bornesup}}
		
		We obtain from  Corollary \ref{cor:uppergene} that the second kind error of the test is controlled by $\beta$ if 
		\begin{align*}
			\| f-f_0\|_2^2 &\geq  \|  f-f_0- \PJ(f-f_0)\|_2^2  +  C(\|f_0\|_2, \|f\|_2, \gamma,\beta) \frac{(\sigma_L^2+1) \sqrt{L}}{n}.
		\end{align*}
		Since $f-f_0 \in \mathcal{B}_{s,2, \infty}(R)$, setting $L=2^J$, we have,  on one hand 
		$$  \|  f-f_0- \PJ(f-f_0)\|_2^2 \leq R^2 2^{-2Js},$$
		and on the other hand,
		$  \|  f \|_2  \leq C(s, R, \|  f_0 \|_2)$. 
		This  leads to the sufficient condition 
		\begin{align*}
			\| f-f_0\|_2^2 &\geq  R^2 2^{-2Js} +  C(s,R, \|f_0\|_2,\gamma,\beta) \frac{(\sigma_L^2+1) 2^{J/2}}{n} .
		\end{align*}
		We recall that $\sigma_L = 2 {\sqrt{2L}}/{\alpha}$. That is, the sufficient condition turns out to be:
		\begin{align}\label{condpuiss}
			&\| f-f_0\|_2^2  \geq  C(s,R, \|f_0\|_2,\gamma,\beta)  \pa{2^{-2Js} + \frac{2^{3J/2}}{\alpha^2 n} +  \frac{2^{J/2}}{ n}}.
		\end{align}
		$J^*$ being set as the smallest integer $J$ such that $2^J \geq (n \alpha^2)^{2/(4s+3)}\wedge n^{2/(4s+1)}$, we consider two cases. 
		\begin{itemize}
			\item If $1/\sqrt n \leq \alpha \leq n^{1/(4s+1)}$, then $ (n \alpha^2)^{2/(4s+3)} \leq n^{2/(4s+1)}$ and the right-hand side of the inequality in Equation~\eqref{condpuiss} for $J= J^*$ is upper bounded  by $$   C(s,R, \|f_0\|_2,\gamma,\beta)  (n \alpha^2)^{-4s/(4s+3)}. $$
			\item  If $ \alpha > n^{1/(4s+1)}$, then $ (n \alpha^2)^{2/(4s+3)} > n^{2/(4s+1)}$ and the right-hand side of the inequality in Equation~\eqref{condpuiss} for $J= J^*$ is upper bounded  by $$   C(s,R, \|f_0\|_2,\gamma,\beta) n^{-4s/(4s+1)}.$$ 
		\end{itemize}
		Hence, the separation rate of our test 
		over the set $  \mathcal{B}_{s,2, \infty}(R)$ is controlled by
		$$ C(s,R, \|f_0\|_2,\gamma,\beta)  \cro{({n} \alpha^2)^{-2s/(4s+3)}  \vee n^{-2s/(4s+1)}}, $$
		which concludes the proof of Theorem \ref{bornesup}.

		\subsection{Adaptivity: proof of Theorem \ref{bornesupadapt}}
		\label{sec:proofAdapt}
		
		In this section, $f_0$ is some fixed density in $\mathbb L_2([0,1])$.
		
		Using the inequality from Equation~\eqref{bornetestagreg},  and the fact that $u_\gamma \geq \gamma/ |\J |$, we obtain that 
		\begin{equation} \label{testagregpuis}
			\mathbb{P}_{Q^n_f}  \pa{ \Delta_{\gamma,Q}^{\J}  =0 } \leq \beta
		\end{equation}
		as soon as 
		$$ \exists J \in \J, \mathbb{P}_{Q^n_{f}} \pa{\tilde {T}_J  \leq \tilde t^{0}_J(1-u_\gamma)} \leq \beta .$$
		We use the result of Corollary~\ref{cor:uppergene}, for $L=2^J$ for some $J \in \J$, where $\sigma_L$ is replaced by $\tilde \sigma_{2^J}$ and $\gamma$ is replaced by $ \gamma/ |\J|$. 
		
		Using the fact that $ |\J| \leq C \log(n)$,  we get that Equation~\eqref{testagregpuis} holds as soon as there exists $J \in \J$ such that
		$$  \| \Pi_{S_{2^J}}(f-f_0)\|^2 \geq   C(\|f_0\|_2, \|f\|_2,\beta)\pa{ \frac{(\tilde \sigma_{2^J}^2 +1) 2^{J/2} }{n \sqrt{\gamma/  |\J|} }},   $$
		or equivalently
		\begin{align*}
			&\| f-f_0\|^2 \geq  \inf_{J \in \J } \Bigg[ \|  f-f_0-  \Pi_{S_{2^J}}  (f-f_0)\|^2 +  C(\|f_0\|_2, \|f\|_2,\gamma,\beta) \frac{(\tilde \sigma_J^2+1) 2^{J/2} \sqrt{\log(n)} }{n} \Bigg] .
		\end{align*}
		Assuming that $f \in \mathcal{B}_{s,2, \infty}(R)$, for some $s>0$ and $R>0$,   we get  that Equation~\eqref{testagregpuis} holds if
		$$
		\| f-f_0\|^2 \geq  \inf_{J \in \J } \Bigg[ R^2 2^{-2Js} +  C(\|f_0\|_2,R,\gamma,\beta)\pa{ 2^{J/2} + \frac{2^{3J/2}\log^2(n)}{\alpha^2}} \frac{\sqrt{\log(n)}}{n}\Bigg] .$$
		Choosing $J \in \J $ as the smallest  integer in $\J$  such that  
		$$
		2^J \geq (n^2 \alpha^4/\log^5(n))^{1/(4s+3)}\wedge (n^2 /\log(n))^{1/(4s+1)},
		$$
		 we obtain the sufficient condition
		$$\| f-f_0\|^2 \geq   C( \|f_0\|_2, R,\gamma,\beta)  \cro{  (n \alpha^2/\log^{5/2}(n))^{-4s/(4s+3)} \vee  (n / \sqrt{\log(n)})^{-4s/(4s+1)}}.$$
		Hence, for all $s>0$, $R >0$, the separation rate of the aggregated test 
		over the set $  \mathcal{B}_{s,2, \infty}(R)$ is controlled by
		$$ C(\|f_0\|_2, R, \gamma,\beta) \Big[ ({n} \alpha^2/\log^{5/2}(n))^{-2s/(4s+3)}   \vee  (n /\sqrt{\log(n)})^{-2s/(4s+1)} \Big], $$
		which concludes the proof of Theorem \ref{bornesupadapt}.

	\end{appendix}
	
	%%%%%%%%%%%%%%%%%%%%%%%%%%%%%%%%%%%%%%%%%%%%%%
	%% Support information (funding), if any,   %%
	%% should be provided in the                %%
	%% Acknowledgements section.                %%
	%%%%%%%%%%%%%%%%%%%%%%%%%%%%%%%%%%%%%%%%%%%%%%
	\section*{Acknowledgements}
	B. Laurent and J-M. Loubes recognize the funding by ANITI ANR-19-PI3A-0004.
	
	%%%%%%%%%%%%%%%%%%%%%%%%%%%%%%%%%%%%%%%%%%%%%%
	%% Supplementary Material, if any, should   %%
	%% be provided in {supplement} environment  %%
	%% with title inside \textbf{} and short    %%
	%% description below.                       %%
	%%%%%%%%%%%%%%%%%%%%%%%%%%%%%%%%%%%%%%%%%%%%%%
	%\begin{supplement}
	%\textbf{???}.
	%???.
	%\end{supplement}
	
	%%  imsart-number.bst  will be used to                     %%
	%%  MR numbers will be added by VTeX.                      %%
	%%                                                         %%
	%%  Use \cite{...} to cite references in text.             %%
	%%                                                         %%
	%%%%%%%%%%%%%%%%%%%%%%%%%%%%%%%%%%%%%%%%%%%%%%%%%%%%%%%%%%%%%
	
	%% if your bibliography is in bibtex format, uncomment commands:
	\bibliographystyle{imsart-number} % Style BST file
	\bibliography{main}       % Bibliography file (usually '*.bib')
	
	%% or include bibliography directly:
	% \begin{thebibliography}{}
	% \bibitem{b1}
	% \end{thebibliography}
	
\end{document}